\documentclass[11pt, leqno]{amsart}
\usepackage[utf8]{inputenc}
\usepackage{graphicx}
\usepackage[lmargin=25mm,rmargin=25mm,tmargin=33mm,bmargin=30mm]{geometry}
\usepackage[english]{babel}

\usepackage{protosem}
\usepackage{xcolor}

\usepackage{amsmath,amsfonts}
\usepackage{amssymb}
\usepackage{amsthm}
\usepackage{mathtools}
\usepackage{csquotes}
\usepackage{mathrsfs}
\usepackage{textcomp}

\makeatletter
\@namedef{subjclassname@2020}{%
  \textup{2020} Mathematics Subject Classification}
\makeatother

\newtheorem{theorem}{Theorem}[section]
\newtheorem{corollary}{Corollary}[theorem]
\newtheorem{lemma}[theorem]{Lemma} 
\newtheorem{proposition}[theorem]{Proposition} 
\newtheorem{remark}{Remark}[theorem]
\numberwithin{equation}{section}
\theoremstyle{definition}
\newtheorem{definition}{Definition}[section]

\newcounter{example}
\newenvironment{example}[1][]{\refstepcounter{example}\par\medskip
   \textbf{Example~\theexample. #1} \rmfamily}{\medskip}

\newcommand{\C}{\ensuremath{{\mathbb C}}}

\author{ NICHOLAS AIDOO}

\address{School of Mathematics, Trinity College Dublin, Dublin 2, Ireland}

\email{aidoon@tcd.ie}



\title{ON THE CATLIN MULTITYPE OF SUMS OF SQUARES DOMAINS}
\date{}

\begin{document}

\maketitle

\noindent ABSTRACT: {\normalsize For a sum of squares domain of finite D'Angelo 1-type at the origin, we show that the polynomial model obtained from the computation of the Catlin multitype at the origin of such a domain is likewise a sum of squares domain. We also prove, under the same finite type assumption, that the multitype is an invariant of the ideal of holomorphic functions defining the domain. Both results are proven using Martin Kol\' a\v r's algorithm for the computation of the multitype introduced in \cite{martinIMRN}. Given a sum of squares domain, we rewrite the Kol\' a\v r algorithm in terms of ideals of holomorphic functions and also introduce an approach that explicitly constructs the homogeneous polynomial transformations used in the algorithm.}

\vspace{1cm}

\tableofcontents

\large
\section{Introduction}
Domains defined by sums of squares of holomorphic functions constitute a very important class in the field of several complex variables as they connect in a very natural way complex analysis with algebraic geometry. This class of domains was introduced by J.J. Kohn in his Acta paper \cite{kohnacta}  under the term \textit{special domains}. Kohn's novel introduction of the technique of subelliptic multiplier ideals to address the local regularity problem paved the way for significant discoveries in several complex variables. In \cite{opendangelo} and \cite{dangelo}, D'Angelo studied the local geometry of real hypersurfaces by assigning to every point on the hypersurface an associated family of ideals of holomorphic functions and exploring various invariants in commutative algebra and algebraic geometry. He established a close connection between the geometry of sums of squares domains and complex algebraic geometry. Further work by Y.-T. Siu in \cite{siunote} and \cite{siunew} on sums of squares domains introduced new approaches for generating multipliers for general systems of partial differential equations. Owing to his initial work on sums of squares domains, Y.-T. Siu gave an extension of the special domain approach to real analytic and smooth cases. S.-Y. Kim and D. Zaitsev in \cite{kimzaitsev} proposed a new class of geometric invariants called \textit{the jet vanishing orders} and used them to establish a new selection algorithm in the Kohn's construction of subelliptic multipliers of sums of squares domains in dimension 3. Also, in a recent paper by the same authors \cite{KZ21}, they provide a solution to the effectiveness problem in Kohn's algorithm for generating multipliers for domains including those defined by sums of squares of holomorphic functions in all dimensions. Other important results pertaining to sums of squares domains can be found in  \cite{fribourgcatda}, \cite{catlincho}, \cite{Fa20}, \cite{khanhzampieri}, and \cite{dimafrancinemartin}.

A key motivation for this paper is to introduce some important tools and techniques necessary for the study of the multitype level set stratification of sums of squares domains. 
We will focus here on the multitype computations for such domains. Our main tool is an algorithm devised by M. Kol\' a\v r in \cite{martinIMRN} for the computation of the Catlin multitype when it has finite entries. In order to ensure this condition is satisfied, we will assume finite D'Angelo 1-type throughout since the latter bounds from above the last entry of the multitype; see \cite{catlinbdry}. 

Owing to the connection between the properties of the associated family of ideals of holomorphic functions and the geometry of sums of squares domains, a natural question is whether or not the multitype could be computed from the corresponding ideals of holomorphic functions. We provide an answer to this question by showing that the multitype of a sum of squares domain can be computed from the related ideal of holomorphic functions via the  restatement of the Kol\' a\v r algorithm at the level of ideals. Besides the fact that working with ideals aligns better with complex algebraic geometry, this restatement also reduces significantly the amount of work involved in computing the Catlin multitype for sums of squares domains.

\medskip A \textit{sum of squares domain} $\Omega \subset \mathbb{C}^{n+1}$ is one whose boundary defining function $r(z)$ is given by 
\begin{equation} \label{sos} 
 r(z) = 2 \text{Re}(z_{n+1}) + \sum_{j=1}^{N} |f_j(z_1, \dots, z_{n+1})|^2,
\end{equation}
where $f_j(z_1, \dots, z_{n+1})$ for all $j, \ 1 \leq j \leq N$ are holomorphic functions vanishing at the origin in $\C^{n+1}$. We shall denote by $M \subset \C^{n+1}$ the hypersurface defined by $\{ z \in \C^{n+1} \ | \ r(z) = 0 \}.$

The \textit{model hypersurface} associated to $M$ at the origin is defined as
\begin{equation} \label{mod}
M_H = \{z \in \C^{n+1} : \ \text{H}(z, \bar z) = 0 \}, 
\end{equation}
the zero locus of the homogeneous polynomial $\text{H}(z, \bar z)$ consisting of all monomials from the Taylor expansion of the defining function that have weight 1 with respect to the multitype weight. We refer to $\text{H}(z, \bar z)$ as the \textit{model polynomial}. 
\medskip

Our first main result concerns the form of the model polynomial $\text{H}(z, \bar z)$:

\begin{theorem} \label{thm1}
Let $M \subset \mathbb{C}^{n+1}$ be a hypersurface whose defining function is given as
\[r(z) = \text{2Re}(z_{n+1}) + \sum_{j=1}^{N}|f_j(z_1, \dots, z_{n+1})|^2,\]
where $f_1, \dots, f_N$ are holomorphic functions defined on a neighborhood of the origin and assume that the D'Angelo 1-type of $M$ at the origin is finite. Then the model associated to $M$ is also a sum of squares 
domain. 
\end{theorem}

In other words, the model polynomial
\begin{equation} \label{modp}
\text{H}(z, \bar z)= 2 \text{Re}(z_{n+1}) + \sum_{j=1}^{N} |h_j(z_1, \dots, z_n)|^2, 
\end{equation}
 where $h_j$ is a polynomial consisting of all terms from the Taylor expansion of $f_j$ of weight $1/2$ with respect to the multitype at the origin, $j=1,\dots, N.$ Note that the corresponding model hypersurface $M_H = \{z \in \C^{n+1} : \ \text{H}(z, \bar z) = 0 \}$ is a decoupled sum of squares domain since variable $z_{n+1}$ has weight 1, so no $h_j$ can depend on it. By Catlin's results in \cite{catlinbdry}, $M_H$ has the same multitype at the origin as the original domain. Therefore, with respect to all multitype computations, sums of squares domains behave as if they were decoupled.

Our second main result answers in the affirmative a question posed by D. Zaitsev during a discussion with the author:

\begin{theorem} \label{thm2}
Let $0\in M \subset \mathbb{C}^{n+1}$ be a hypersurface whose defining function is given as
\[r(z) = \text{2Re}(z_{n+1}) + \sum_{j=1}^{N}|f_j(z_1, \dots, z_{n+1})|^2,\]
where $f_1, \dots, f_N$ are holomorphic functions defined on a neighborhood of the origin and assume that the D'Angelo 1-type of $M$ is finite at the origin. 

Then the Catlin multitype is an invariant of the ideal of holomorphic functions defining the domain.
\end{theorem}

Computing the multitype via the Kol\' a\v r algorithm requires one to use weighted homogeneous polynomial transformations at each step in order to generate an intermediate weight and a partial model polynomial that depends on a minimal number of variables. How these homogeneous polynomial transformations are obtained was not explicitly described in \cite{martinIMRN}, however. In this paper, we characterize these weighted homogeneous polynomial transformations and explicitly construct them by relating them to elementary row and column operations on the Levi matrix associated to the domain. More specifically, we devise an algorithm for the construction of these homogeneous polynomial transformations, the row reduction algorithm. We call a polynomial change of variables at step $j$ of the Kol\' a\v r algorithm {\it allowable with respect to variable $z_k$} if it eliminates the occurrence of variable $z_k$ from the leading polynomial $P_j$ and reduces the number of variables in $P_j$ by at least one. For the correspondence between polynomial changes of variables and elementary row and column operations on the Levi matrix associated to the domain, we introduce a variant of the notion of dependence compared to the standard one in linear algebra as follows: Let $\text{A}_{P_j}$ be the Levi matrix of the leading polynomial $P_j$ at step $j$ of the Kol\' a\v r algorithm. For a given $k,$ denote by $R_k$ and $C_k$ the $k$-th row and the $k$-th column of the matrix $\text{A}_{P_j}$ respectively. Let $\mathscr{R}$ be the set of all rows of the matrix  $\text{A}_{P_j}.$ We call a row $R_k$ of $\mathscr{R}$  \textit{dependent} if it satisfies the condition:

\[
R_k = \sum_{\substack{l=1 \\
            l \neq k}}^n \alpha_l R_l,  
\]
where $\alpha_l \in \C[z], $ $\alpha_l \neq 0$ for at least one $l.$ In other words, the row $R_k$ is dependent if it can be written as a linear combination of the other rows with polynomial coefficients. The set $\mathscr{R}$ is said to be \textit{dependent} if at least one of the rows is dependent.

\smallskip Since the matrix $\text{A}_{P_j}$ is Hermitian, a similar definition holds for the $k$-th column $ C_k $ of $\text{A}_{P_j} $ namely $$ C_k = \sum_{\substack{l=1\\
           l \neq k}}^n \bar{\alpha}_l C_l.$$ 

\bigskip\noindent Hence, we prove the following:

\begin{proposition} \label{P1}
Assume that the D'Angelo 1-type of the hypersurface $M$ in $\C^{n+1}$ at 0 is finite. At step $j$ of the Kol\' a\v r algorithm for the computation of the 
multitype at 0, let $P_j$ be the leading polynomial, and let $Q_j$ be the leftover polynomial.
 Let $k$ be given, $1 \leq k \leq n.$ There exists an allowable polynomial transformation on $P_j$ with respect to the variable $z_k$ if and only if the $k$-th row of $\text{A}_{P_j}$ is dependent. 
\end{proposition}  

Using this characterization of the polynomial transformations and the restatement of the Kol\' a\v r algorithm in terms of ideals of holomorphic functions,  the row reduction algorithm connects nicely the notion of simplifying the Jacobian module associated to a sum of squares domain with elementary row operations on the complex Jacobian matrix of the same domain. By employing this algorithm at every step of the Kol\' a\v r algorithm for the computation of the multitype, we are able to construct the weighted homogeneous polynomial transformations needed in the Kol\' a\v r algorithm.
 
This paper is structured in the following manner: Section 2 describes some notation and provides definitions that are pertinent to our discussions in subsequent sections. Section 3 defines the Catlin multitype and provides a thorough description of the Kol\' a\v r algorithm as introduced in \cite{martinIMRN}. Section 4 presents a key lemma for the characterization of the multitype entries of the sum of squares domain. Specifically, we establish the fact that each multitype entry can be realized by the modulus square of some monomial of vanishing order at least one in the multivariate ring of polynomials over $\C.$  Section 5 contains the proofs of Theorems \ref{thm1} and \ref{thm2}. Section 6 presents a modified version of the Kol\' a\v r algorithm in terms of ideals of holomorphic polynomials. In the same section, a corollary to Theorem \ref{thm1} is stated and proven. Finally, section 7 contains the proof of Proposition \ref{P1} and also provides a detailed description of the row reduction algorithm for the construction of the weighted homogeneous polynomials required in the Kol\' a\v r algorithm. 

\medskip
This article constitutes part of the author's Ph.D thesis at Trinity College Dublin under the supervision of Andreea Nicoara.

\medskip
\noindent \textbf{Acknowledgements}: Special thanks goes to the author's supervisor Andreea Nicoara for all her valuable remarks and guidance during our many discussions. The author would also like to thank Dmitri Zaitsev for his useful comments and Berit Stens\o nes for her encouragement.

\vspace{0.3cm}

\section{Definitions and Notation}
We present some definitions we use in the article following the setup of Kol\' a\v r in \cite{martinIMRN}. Let $M$ be a hypersurface in $\C^{n+1}$ and $p \in M$ be a Levi degenerate point. We will assume that $p$ is a point of finite D'Angelo 1-type. Let $(z,w)$ be local holomorphic coordinates centered at the point $p$, where $w = u +iv$ is the complex non-tangential variable and the complex tangential variables are in the n-tuple $z = (z_1, \dots, z_n)$ with $z_k = x_k + i y_k$. Throughout this paper, we will compute and define weights by considering only the complex tangential variables $z_1, \dots, z_n$ as in \cite{martinIMRN}.

\begin{definition} \label{def2}
A weight $\Lambda = (\mu_1, \dots, \mu_n)$ is an n-tuple of rational numbers with $0 \leq \mu_j \leq \frac{1}{2}$ satisfying:
\begin{itemize}
\item[i.] $\mu_j \geq \mu_{j+1}$ for $1 \leq j \leq n-1;$ 
\item[ii.] For each $t$, either $\mu_t = 0$ or there exists a sequence of nonnegative integers $a_1, \dots, a_t$ satisfying $a_t > 0$ such that
\[\sum_{j=1}^t a_j \mu_j = 1.\]
\end{itemize}
\end{definition}

 Let $\Lambda$ be a weight. If $\alpha = (\alpha_1, \dots, \alpha_n)$ is a multiindex, then we define the weighted length of $\alpha$ by
\[|\alpha|_{\Lambda} =  \sum_{j=1}^n \alpha_j \mu_j.\] 
Also if $\alpha = (\alpha_1, \dots, \alpha_n)$ and $ \hat{\alpha} = (\hat{\alpha}_1, \dots, \hat{\alpha}_n)$ are multiindices then the weighted length of the pair $(\alpha, \hat{\alpha})$ is defined by
\[|(\alpha, \hat{\alpha})|_{\Lambda} = \sum_{j=1}^n (\alpha_j + \hat{\alpha}_j) \mu_j.\]

\begin{definition} \label{def3}
A monomial $A_{\alpha \hat{\alpha} l} z^{\alpha} z^{\hat{\alpha}} u^l$ is said to be of weighted degree $\kappa$ if 
\[\kappa: = l + |(\alpha, \hat{\alpha})|_{\Lambda}.\]
Similarly, we define the weighted order of the differential operator $D^{\alpha}\bar{D}^{\hat{\alpha}}D^l$ to equal to $\kappa: = l + |(\alpha, \hat{\alpha})|_{\Lambda}$, where 
\[D^{\alpha} = \frac{\partial^{|\alpha|}}{\partial z_1^{\alpha_1} \cdots \partial z_n^{\alpha_n}}, \ \ \bar{D}^{\hat{\alpha}} = \frac{{\partial}^{\hat{|\alpha|}}}{{\partial} \bar{z}_1^{\hat{\alpha}_1} \cdots {\partial} \bar{z}_n^{\hat{\alpha}_n}}, \ \ \text{and} \ \ D^l = \frac{{\partial}^l}{\partial u^l}.\]

\medskip \noindent A polynomial $P(z, \bar{z}, u)$ is said to be $\Lambda$-homogeneous of weighted degree $\kappa$ if it is a sum of monomials of weighted degree $\kappa$.
\end{definition}
We shall set the variable $w$ as well as the variables $u \ \text{and} \ v$ to have a weight of one. 

\begin{definition} \label{def4}
A weight $\Lambda = (\mu_1, \dots, \mu_n)$ is said to be \textit{distinguished} if there exist local holomorphic coordinates $(z,w)$ mapping $p$ to the origin such that the boundary-defining equation for $M$ in the new coordinates is of the form
\begin{equation} \label{def8}
v= P(z, \bar{z}) + o_{\Lambda}(1),
\end{equation}
where $P(z, \bar{z})$ is a $\Lambda$-homogeneous polynomial of weighted degree 1 without pluriharmonic terms and $o_{\Lambda}(1)$ denotes a smooth function whose derivatives of weighted order less than or equal to 1 vanish at zero.

\end{definition}

\medskip We order the weights lexicographically. This means that for the pair of weights $\Lambda_1 = (\mu_1, \dots, \mu_n)$ and $\Lambda_2 = (\mu_1^{'}, \dots, \mu_n^{'})$, $\Lambda_1 > \Lambda_2$ if for some $t$, $\mu_j = \mu_j^{'}$ for $j <t$ and $\mu_t> \mu_t^{'}$. 

\begin{definition} \label{def5}
Let $\Lambda = (\lambda_1, \dots, \lambda_n)$ be a weight and 
\[
\tilde{w} = w + g(z_1, \dots, z_n, w) \:\: \text{and} \:\: \tilde{z}_j = z_j + f_j(z_1, \dots, z_n, w),
\] 
for $1 \leq j \leq n,$ be a holomorphic change of variables. We say that this transformation is:

\begin{itemize}
\item[i.] $\Lambda$-homogeneous if $f_j$ is a $\Lambda$-homogeneous polynomial of weighted degree $\lambda_j$ and $g$ is a $\Lambda$-homogeneous polynomial of weighted degree 1,
\item[ii.] $\Lambda$-superhomogeneous if $f_j$ has a Taylor expansion consisting of monomials that have weighted degree $\geq \lambda_j$ and $g$ consists of terms of weighted degree $\geq 1$,
\item[iii.] $\Lambda$-subhomogeneous if the Taylor expansion of $f_j$ consists of terms of weighted degree $\leq \lambda_j$ and $g$ consists of weighted degree $\leq 1$.
\end{itemize}
\end{definition}

J. J. Kohn introduced the notion of type of a point on a pseudoconvex hypersurface in $\C^2$ in \cite{kboundary}. In \cite{bloomgraham} Thomas Bloom and Ian Graham generalized Kohn's notion to $\C^n$ and gave a geometric characterization of type of points on real hypersurfaces in $\C^n$.
Here is the definition given by Bloom and Graham: Let $\mathcal{N}$ be a real $C^\infty$ hypersurface defined in an open subset $U \subset \C^n$ with defining function $r.$ Let $\mathscr{L}_k$ for $k \geq 0$ an integer, be the module, over $C^{\infty}(U)$, of vector fields generated by the tangential holomorphic vector fields to $\mathcal{N},$ their conjugates, and commutators of order less than or equal to $k$ of such vector fields.

\begin{definition} \label{def1}
A point $p \in \mathcal{N}$ is of type $m$ if $\big < \partial r(p), F(p) \big > = 0$ for all $F \in \mathscr{L}_{m-1}$ while $\big < \partial r(p), F(p) \big > \neq 0$ for some $F \in \mathscr{L}_{m}$.
\end{definition}

Here we denote the contractions between a cotangent vector and a tangent vector by $\big < , \big >$.  We shall refer to the type at a point $p$ as defined above as the \textit{Bloom-Graham type}.

\section{Catlin Multitype and the Kol\' a\v r Algorithm}
In this section, we describe the Catlin multitype and the Kol\' a\v r algorithm for the computation of the multitype at the origin  in \cite{martinIMRN}.
The notion of weights, distinguished weights, and the multitype was introduced by D. Catlin in \cite{catlinbdry}.  We will follow the notation and definitions as given by Kol\' a\v r in \cite{martinIMRN} and further describe some of the tools that M. Kol\' a\v r introduced:

\begin{definition} \label{defa}
Let $M$ be a hypersurface in $\C^{n+1},$ and let $p \in M.$ Let $\Lambda^{*}=(\mu_1,\dots,\mu_n)$ be the greatest lower bound with respect to the lexicographic ordering of all the distinguished weights at $p$. The \textit{multitype} $\mathscr{M}$ at $p$ is defined to be the n-tuple $(m_1,\dots, m_n)$, where $m_j = \infty$ if $\mu_j =0$ and $m_j= \frac{1}{\mu_j}$ if $\mu_j \neq 0$. We call the multitype $\mathscr{M}$ at $p$ finite, if the last entry $m_n< \infty$.
\end{definition}

\medskip\noindent The next theorem from \cite{catlinbdry} clarifies the relationship between the multitype and the D'Angelo type for a pseudoconvex domain:

\begin{theorem}{\normalfont (Catlin)}.\label{catlinthm}
Let $\Omega \subset \mathbb{C}^{n+1}$ be a pseudoconvex domain smooth boundary. Let $p_0\in b\Omega$ be a boundary point. If $\mathscr{M}(p_0)=(m_1,\dots, m_n)$ is the multitype at $p_0$, then for each $q=1, \dots, n,$ $m_{n+1-q} \leq \Delta_q(b\Omega, p_0),$ where $ \Delta_q(b\Omega, p_0)$ is the D'Angelo q-type at $p_0$.
\end{theorem}
 
\medskip For the purposes of this paper, we shall assume finite D'Angelo 1-type at any point $p \in M,$ since this assumption by Theorem \ref{catlinthm} ensures that all the entries of the multitype are finite, which is the exact setting in which the Kol\' a\v r algorithm works. 

\medskip For a weight $\Lambda$, we say the local coordinates on $M$ at $p$ are $\Lambda$-adapted if $M$ is described locally to have the form in \eqref{def8}, where $P$ is $\Lambda$-homogeneous. We shall refer to $\Lambda^{*}$-adapted coordinates as the multitype coordinates given such that $P$ is $\Lambda^*$-homogeneous.

\medskip Let $\gamma_j$, $j=1, \dots, c$, be the length of the $j$-th constant piece of the multitype weight given such that $c$ is the number of distinct entries in the multitype. Let $\sum_{i=1}^j \gamma_i = k_j$, then we have
\[
\mu_1= \cdots =\mu_{k_1} > \mu_{k_1+1} = \cdots = \mu_{k_2} > \cdots = \mu_{k_{c-1}} > \cdots = \mu_{k_{c-1}+1} = \cdots = \mu_n,
\]
where $n = k_c.$ We define a monotone sequence of weights $\Lambda_1, \dots, \Lambda_c$ which are ordered lexicographically as follows. $\Lambda_1$ is a constant n-tuple $(\mu_1, \dots, \mu_1)$ and $\Lambda_c = \Lambda^{*}$ is the multitype weight. We then define the weight $\Lambda_j = (\lambda_1^j, \dots, \lambda_n^j)$ for  $1 < j <c$, by $\lambda_i^j = \mu_i$ for $i \leq k_{j-1}$ and $\lambda_i^j = \mu_{k_{j-1}+1}$ for $i > k_{j-1}$. Note that this construction yields a finite sequence of  weights even if $\Lambda^*$ has some infinite entries.

\begin{definition} \label{defb}
Fix $\Lambda^{*}$-adapted local coordinates. The {\it leading polynomial} $P$ is defined as

\begin{equation} \label{eq:7}
P(z, \bar{z}) = \sum_{|(\alpha, \hat{\alpha})|_{\Lambda^{*}} = 1} C_{\alpha,\hat{\alpha}} z^{\alpha} \bar{z}^{\hat{\alpha}}.
\end{equation}  
\end{definition}

The polynomial defined in \eqref{eq:7} is exactly the polynomial that only retains the terms of weight 1. Put differently, it is a $\Lambda^{*}$-homogeneous  polynomial of weighted degree 1 with no pluriharmonic terms, where $\Lambda^{*}$ is the multitype weight. Following Kol\' a\v r in \cite{martinIMRN}, we will also denote by {\it leading polynomial} the polynomial consisting of all terms of weight 1 with respect to each intermediate weight $\Lambda_j$ in the Kol\' a\v r algorithm.

\begin{theorem} {\normalfont (Kol\' a\v r)}. \label{t1}
A biholomorphic transformation takes $\Lambda^{*}$-adapted coordinates into $\Lambda^{*}$-adapted coordinates if and only if this transformation is $\Lambda^{*}$-superhomogeneous.
\end{theorem}

We will now describe the Kol\' a\v r algorithm for the computation of the multitype and apply this theorem under the assumption that all entries of the multitype are finite.

\vspace{0.5cm}

\noindent{\textbf{The Kol\' a\v r Algorithm:} We start by considering local holomorphic coordinates in which the leading polynomial in the variables $z$ and $\bar{z}$ contains no pluriharmonic terms. The degree of the lowest order monomial in this polynomial is then equal to the Bloom-Graham type of $M$ at $p$ as defined in \cite{bloomgraham}. This gives the first multitype component $m_1;$ see \cite{catlinbdry}. By our assumption $1 < m_1 < +\infty.$ Let $m_1 = \frac{1}{\mu_1}$ and set $\Lambda_1 = (\mu_1, \dots, \mu_1)$. We then consider all $\Lambda_1$-homogeneous transformations and choose one that will make the leading polynomial $P_1$ to be independent of the largest number of variables. We denote this number by $d_1.$ So for such coordinates we get  the defining function of $M$ to be of the form 
\[v = P_1(z_1, \dots, z_{n-d_1}, \bar{z}_1, \dots, \bar{z}_{n-d_1}) + Q_1(z, \bar{z}) + o(u),\]
where $P_1$ is $\Lambda_1$-homogeneous of weighted degree 1 and $Q_1$ is $o_{\Lambda_1}(1)$. We use the result that any weight $\Lambda$ which is smaller than $\Lambda_1$ with respect to the lexicographic ordering and the fact that $\Lambda_1$-homogeneous transformation are linear, to conclude that $\Lambda$-adapted coordinates are also $\Lambda_1$-adapted coordinates. We thus have that $\mu_1 = \cdots = \mu_{n-d_1}$ and $\mu_{n-d_1+1} < \mu_1$. We define the following important tools:

$$\Theta_1 = \left\{(\alpha, \hat{\alpha})| \ C_{\alpha, \hat{\alpha}}^1 \neq 0 \:\: \text{and} \:\: \sum_{i=1}^{n-d_1} (\alpha_i + \hat{\alpha}_i) \mu_i < 1 \right\},$$
and also

\[Q_1(z, \bar{z}) = \sum_{|(\alpha, \hat{\alpha})|_{\Lambda_1} > 1} C_{\alpha,\hat{\alpha}}^1 z^{\alpha} \bar{z}^{\hat{\alpha}}.\]
For every $(\gamma, \hat{\gamma}) \in \Theta_1$,

\begin{equation} \label{eq:8}
W_1(\gamma, \hat{\gamma}) = \frac{1 - \sum_{i=1}^{n-d_1}(\gamma_i + \hat{\gamma}_i) \mu_i}{\sum_{i=n-d_1+1}^{n}(\gamma_i + \hat{\gamma}_i)}.
\end{equation}
We define the next weight $\Lambda_2$ by letting 

\[\lambda_j^2 = \max_{(\alpha, \hat{\alpha}) \in \Theta_1} \text{W}_1(\alpha,\hat{\alpha})\] 
for $j > n-d_1$ and $\lambda_j^2= \mu_1$ for $j \leq n-d_1$. We then complete the second by letting $P_2$ be the new leading polynomial corresponding to the weight $\Lambda_2$. $P_2$ depends on more than $n-d_1$ variables.

We proceed by induction. At the $j$-th step, for $j>2,$ using coordinates from the previous step, we consider all $\Lambda_{j-1}$-homogeneous transformations and choose one that makes the leading polynomial $P_{j-1}$ to be independent of the largest number of variables. We fix such coordinates, and let $d_{j-1}$ be the largest number of variables, which do not show up in $P_{j-1}$ after this change of variables. By Theorem \ref{t1}, the transformations taking $\Lambda_{j-1}$-adapted coordinates into $\Lambda_{j-1}$-adapted coordinates are always $\Lambda_{j-1}$-superhomogeneous. The number of multitype entries that are added at each step of the computation depends on the difference $(d_{j-2} - d_{j-1})$. Hence we consider two cases in this step:

\begin{itemize}
\item[CASE \Large 1:] Assume that $d_{j-2} > d_{j-1}$. Also recall that for any weight $\Lambda$ that is smaller than $\Lambda_{j-1}$ with respect to the lexicographic ordering, $\Lambda$-adapted coordinates are also $\Lambda_{j-1}$-adapted. This implies that we get $(d_{j-2} - d_{j-1})$ multitype entries
\[
\mu_{n-d_{j-2}+1} = \cdots = \mu_{n-d_{j-1}} = \lambda_{n- d_{j-2}+1}^{j-1}
\]
and let $\lambda_i^j=\mu_i$ for $i \leq n - d_{j-1}- 1$. To obtain $\lambda_i^j$ for $j>n-d_{j-1}-1$, we consider
\[
v = P_{j-1}(z_1, \dots, z_{n-d_{j-1}}, \bar{z}_1, \dots, \bar{z}_{n-d_{j-1}}) + Q_{j-1}(z, \bar{z}) + o(u),
\]
where $Q_{j-1}$ is $o_{\Lambda_{j-1}}(1)$ and $P_{j-1}$ is $\Lambda_{j-1}$-homogeneous of weighted degree 1. We define $Q_{j-1}$, $\Theta_{j-1}$, and $W_{j-1}$ in a similar way as in step two. Thus,

$$\Theta_{j-1} = \left\{(\alpha, \hat{\alpha})| \ C_{\alpha, \hat{\alpha}}^{j-1} \neq 0 \:\: \text{and} \:\: \sum_{i=1}^{n-d_{j-1}} (\alpha_i + \hat{\alpha}_i) \mu_i < 1 \right\},$$
and also

\[Q_{j-1}(z, \bar{z}) = \sum_{|(\alpha, \hat{\alpha})|_{\Lambda_{j-1}} > 1} C_{\alpha,\hat{\alpha}}^{j-1} z^{\alpha} \bar{z}^{\hat{\alpha}}.\]
For every $(\gamma, \hat{\gamma}) \in \Theta_{j-1}$,

\begin{equation} \label{eq:9}
W_{j-1}(\gamma, \hat{\gamma}) = \frac{1 - \sum_{i=1}^{n-d_{j-1}}(\gamma_i + \hat{\gamma}_i) \mu_i}{\sum_{i=n-d_{j-1}+1}^{n}(\gamma_i + \hat{\gamma}_i)}.
\end{equation}
So for the remaining multitype entries of $\Lambda_j$ we let

\[\lambda_i^j = \max_{(\alpha, \hat{\alpha}) \in \Theta_{j-1}} W_{j-1}(\alpha,\hat{\alpha}),\] 
for $i > n-d_{j-1}$. 

\item[CASE \Large 2:] Assume that $d_{j-1} = d_{j-2}$. There are zero multitype entries computed in this case and so we only determine $\lambda_i^j$ for $j > n-d_{j-1}$ using \eqref{eq:9}. This completes the $j$-th step of the computation. 
\end{itemize}

The process terminates after a finite number of steps to give all the entries of the multitype weight $\Lambda^{*}$. It is clear that case 1 advances the process. We just need to show that the number of times case 2 occurs where no multitype entries are determined can only happen finitely many times. We claim case 2 can take place at most $\lceil \frac{1}{\mu_n} \rceil^{n - d_{j-1} + 1}$ times, where $\lceil \frac{1}{\mu_n} \rceil$ is the ceiling for the rational number $\frac{1}{\mu_n}.$ Indeed, it comes down to the number of different values that \eqref{eq:9} can have. The upper bound for the numerator is given by $\lceil \frac{1}{\mu_n} \rceil^{n - d_{j-1}}$ as the $\mu_i$ entries are decreasing, whereas the upper bound for the denominator is given by  $\lceil \frac{1}{\mu_n} \rceil.$

\vspace{0.5cm}

\begin{example}\label{K1}
Let $M \subset \C^4$ be a real hypersurface near the origin given by $r =0$ with

$$r = 2\,\text{Re}(z_4 ) + |z_1-z_2+z_3^2|^2 + |z_1^2-z_2^2|^2.$$ 

\vspace{0.3cm}\noindent Using the Kol\' a\v r algorithm, we proceed as follows: The Bloom-Graham type is 2, which implies that $\mu_1 = \frac{1}{2}$ \ and \ $\Lambda_1 = (\frac{1}{2}, \frac{1}{2}, \frac{1}{2})$. Thus, $P_1=|z_1-z_2|^2.$ We consider all $\Lambda_1$-homogeneous transformation and choose
\[
\tilde{z}_1 = z_1 - z_2, \quad  \text{and} \quad \tilde{z}_j = z_j,
\]
for $j = 2, 3, 4,$ to obtain a leading polynomial $P_1$ independent of the variables $z_2$ and $z_3$. We write $r$ in the new variables and ignore $\sim$ when no confusion arises. Then $P_1 = |z_1|^2$ with $d_1 = 2.$ The leftover polynomial is
\[
Q_1 = |z_3^2|^2 + 2\text{Re}(z_1 \bar{z}_3^2) + |z_1^2|^2 + 4\text{Re}(z_1^2 \bar{z}_1 \bar{z}_2) +4|z_1z_2|^2.
\]
$\Lambda_2 = (\frac{1}{2}, \frac{1}{4}, \frac{1}{4}),$ where the maximum number $\max(W_1) = \frac{1}{4}$.  We consider all $\Lambda_2$-homogeneous transformations and choose
\[
\tilde{z}_1 = z_1 + z_3^2  \ \ \:\: \text{and} \:\: \ \ \tilde{z}_j = z_j  
\]
for $j = 2,3,4,$ which makes the leading polynomial independent of the variables $z_2$ and $z_3$. In the new variables, $P_2 = |z_1|^2$ with  $d_2 = 2$ and  
\[
\begin{split}
Q_2 = & |z_1^2|^2 + 4|z_1z_3^2|^2 + 4\text{Re}(z_1^2 \bar{z}_1\bar{z}_3^2) + |z_3^4|^2 + 2\text{Re}(z_1^2 \bar{z}_3^4) + 4\text{Re}(z_1z_3^2 \bar{z}_3^4) + 4|z_1z_2|^2 + 4\text{Re}(z_1^2 \bar{z}_1\bar{z}_2) \\
  & + 8\text{Re}(z_1z_3^2 \bar{z}_1\bar{z}_2)+ 4\text{Re}(z_3^4 \bar{z}_1\bar{z}_2) + 4|z_2z_3^2|^2 - 4\text{Re}(z_1^2 \bar{z}_2\bar{z}_3^2) - 8\text{Re}(z_1z_3^2 \bar{z}_2\bar{z}_3^2) \\
 & - 4\text{Re}(z_3^4 \bar{z}_2\bar{z}_3^2)  - 8\text{Re}(z_1z_2 \bar{z}_2\bar{z}_3^2).
 \end{split}
\]
$\Lambda_3 = (\frac{1}{2}, \frac{1}{6}, \frac{1}{6})$ with $\max (W_2)= \frac{1}{6}.$ Here, no $\Lambda_3$-homogeneous transformation can make $P_3$ to be independent of any variables and so $d_3 = 0.$ Hence,
 \[
P_3 = |z_1|^2 + 4|z_2z_3^2|^2.
\]
\end{example}
  
\section{Squares of Monomials and the Multitype Entries}
We will show in this section that the set of monomials that give the maximum W-value at each step of the Kol\' a\v r algorithm always consists of both squares of moduli of monomials and cross terms. Since the entries in the multitype depend on the maximum W-values, it will suffice to establish that for each entry of the multitype, there is always a square that gives the corresponding multitype entry.

As hinted in the introduction, we rely on the Kol\' a\v r algorithm for the computation of the multitype in \cite{martinIMRN} to develop tools to study the stratification by multitype level sets of the boundary of a domain given by a sum of squares of holomorphic functions. We seek a tool that will enable us to effectively interpret our results both geometrically and algebraically. As such, we resort to reducing the problem to the study of the ideal of the holomorphic functions in the sum. For the transition to be effective, we need to establish that a natural modification of the Kol\' a\v r algorithm still holds at the level of ideals.

The ensuing lemma gives the foundational result we need in order to transition from the sums of squares case to the case of ideals of holomorphic functions in $\C^n.$ We use the result from this lemma to prove Theorems \ref{thm1} and \ref{thm2}.

\begin{lemma} \label{nl}
Let $f$ and $g$ be monomials with non-zero coefficients from the Taylor expansion of $h,$ where $h$ is a generator from a sum of squares domain in $\C^{n+1}$. Let $P_t$ for $t \geq 1$ be the leading polynomial at step $t$ of the Kol\' a\v r algorithm, and let $W_t$ be the number computed at the $(t+1)$-th step. 

\begin{itemize}
\item[A.] If $W_t(|f|^2) = W_t(|g|^2)$, then $W_t(f\bar{g})= W_t(|f|^2) = W_t(|g|^2)$.
\item[B.] If $W_t(|f|^2) < W_t(|g|^2)$, then $ W_t(|f|^2) < W_t(f\bar{g}) < W_t(|g|^2)$.
\item[C.] If $W_t(|f|^2)$ cannot be computed, then 
\begin{itemize}
\item[i.] $W_t(f\bar{g}) \leq W_t(|g|^2)$ for any monomial $g$ for which both $W_t(|g|^2)$ and $W_t(f\bar{g})$ can be computed.
\item[ii.] $W_t(f\bar{g})$ cannot be computed for any monomial $g$ for which $W_t(|g|^2)$ cannot be computed.
\end{itemize}
\item[D.] If $f$ is such that $|f|^2$ is in the leading polynomial $P_t$, then
\begin{itemize}
\item[i.] For any monomial $g$ for which $W_t(|g|^2)$ can be computed, $W_t(f\bar{g})= W_t(|g|^2)$.
\item[ii.] For any monomial $g$ for which $W_t(|g|^2)$ cannot be computed, $W_t(f\bar{g})$ cannot be computed as well.
\end{itemize}
\end{itemize}
Here $W_t(f):= W_t(\alpha,\hat{\alpha})$ where $(\alpha, \hat{\alpha})$ is the pair of multiindices corresponding to the monomial $f$.
\end{lemma}

\begin{remark} \label{rm1} 
We shall say the quantity $W_t(f)$ ``\textit{cannot be computed}'' if the pair of multiindices $(\alpha, \hat{\alpha})$ corresponding to the monomial $f$ is not an element of $\Theta_t.$ In other words, $W_t(f)$ cannot be computed if the numerator of the fraction giving $W_t(f)$ is not positive; see \eqref{nr0} below.

\end{remark}

\begin{remark} \label{rm2}
 We note here that $W_t(f\bar{g})$ and $W_t(\bar{f} g)$ corresponding to the cross terms $f \bar g$ and $\bar f g$ respectively are equal.
\end{remark}

\begin{proof}
Let $z_1, \dots, z_c$ be the variables in the leading polynomial $P_t$, and let the weight $\Lambda_t = (\mu_1, \dots,\mu_c, \mu_{c+1}, \dots,\mu_n)$. 
We begin by recalling that 
\begin{equation} \label{nr0}
W_t(\alpha, \hat{\alpha}) = \frac{1 - \sum_{i=1}^{c}(\alpha_i+\hat{\alpha}_i)\mu_i}{\sum_{i=c +1}^n(\alpha_i+\hat{\alpha}_i)},
\end{equation} 
where $(\alpha, \hat{\alpha}) = (\alpha_1, \dots, \alpha_n, \hat{\alpha}_1,\dots, \hat{\alpha}_n)$ is the multiindex of the monomial whose $W_t$ is being computed.

Let $\Gamma_1$ be the set of all non-zero monomials that consist of only variables not in $P_t$, let $\Gamma_2$ be the set of all non-zero monomials which consist of variables both in $P_t$ as well as variables not in $P_t$, and let $\Gamma_3$  be the set of all non-zero monomials which consist of only variables in $P_t$. If $W_t(|f|^2)$ can be computed, then $f \in \Gamma_1$ or $f \in \Gamma_2$ only. We will now prove $f$ cannot belong to $\Gamma_3.$ We assume the opposite, namely that $f \in \Gamma_3$ and that $W_t(|f|^2)$ can be computed. The monomial $|f|^2$ has weight $\geq 1$ with respect to $\Lambda_t,$ and since $f \in \Gamma_3,$ it follows that the numerator of $W_t(|f|^2)$ is $\leq 0.$ Therefore, $W_t(|f|^2)$ cannot be computed, which is a contradiction.  Without loss of generality, we now specify to $f \in \Gamma_1$  or $f \in \Gamma_2$  for all parts of the lemma pertaining to the case when $W_t(|f|^2)$ can be computed (and similarly for $g$).

Since $f \in \Gamma_1$  or $f \in \Gamma_2$ , we can write $f = f_1\, f_2$ where $f_1$ and $f_2$ are monomials satisfying $f_1 \in$ $\Gamma_3$ and $f_2 \in\Gamma_1$. Let $f_1 = z_1^{\alpha_1} \cdots z_c^{\alpha_c},$ where $z_1, \dots, z_c$ are the variables in the leading polynomial $P_t$, and $f_2= C_f z_{c+1}^{\beta_{c+1}} \cdots z_n^{\beta_n}$ for $C_f \in \C.$ If $W_t(|f|^2)$ can be computed, by our definition of $f$, the multiindices $\alpha$ and $\beta$ corresponding to monomials $f_1$ and $f_2$ respectively must satisfy $|\alpha| \geq 0$ and $|\beta| > 0$. If $|\alpha| =0$, then $f \in \Gamma_1,$ and if $|\alpha| >0$, then $f \in$ $\Gamma_2.$ Here $|\alpha|= \alpha_1 + \cdots + \alpha_c$ as the rest of the entries are zero, and $|\beta|=\beta_{c+1}+\dots + \beta_n$ for the same reason. Now, $|f|^2 = f_1\bar{f}_1 f_2\bar{f}_2$ with $|\alpha|=|\hat{\alpha}|$ and $|\beta| = |\hat{\beta}|$. Hence
\begin{equation} \label{nr1}
W_t(|f|^2) = \frac{1 - \sum_{i=1}^{c}(\alpha_i+\hat{\alpha}_i)\mu_i}{|\beta| + |\hat{\beta}|} = \frac{\frac{1}{2} - \sum_{i=1}^c \alpha_i\mu_i}{|\beta|}. 
\end{equation}

Similarly, let $g = g_1g_2$, where $g_1 \in \Gamma_3$ and $g_2 \in \Gamma_1.$ Let $\gamma$ and $\tau$ be the multiindices corresponding to monomials $g_1$ and $g_2$ respectively. Here $g_1 = z_1^{\gamma_1} \cdots z_c^{\gamma_c}$ and $g_2= C_g z_{c+1}^{\tau_{c+1}} \cdots z_n^{\tau_n}$ for $C_g \in \C.$ By a similar computation as carried out above for $f,$
\begin{equation} \label{nr2}
W_t(|g|^2) = \frac{\frac{1}{2} - \sum_{i=1}^c\gamma_i\mu_i}{|\tau|}.
\end{equation}

\begin{itemize}
\item[A.] Suppose that $W_t(|f|^2) = W_t(|g|^2)$ and consider

\begin{equation} \label{nr3}
\begin{split}
W_t(f\bar{g}) & = \frac{1 - \sum_{i=1}^{c}\alpha_i \mu_i - \sum_{i=1}^c \hat{\gamma}_i \mu_i}{|\beta| + |\tau|} \\
 & = \frac{\frac{1}{2} - \sum_{i=1}^{c}\alpha_i \mu_i + \frac{1}{2} - \sum_{i=1}^c \gamma_i\mu_i}{|\beta| + |\tau|}\\
 & = \frac{|\beta| W_t(|f|^2)+ |\tau| W_t(|g|^2)}{|\beta| + |\tau|} \ \quad \text{from \eqref{nr1} and \eqref{nr2} by cross multiplication.}\\
 & = W_t(|f|^2) = W_t(|g|^2) \ \quad \text{by our hypothesis.}
\end{split} 
\end{equation}

\item[B.] Suppose that $W_t(|f|^2) < W_t(|g|^2)$. Then from (A.) we get that 
\begin{equation} \label{nr4}
\begin{split}
W_t(f\bar{g}) & = \frac{\frac{1}{2} - \sum_{i=1}^{c}\alpha_i \mu_i + \frac{1}{2} - \sum_{i=1}^c \gamma_i\mu_i}{|\beta| + |\tau|}\\
 & = \frac{|\beta| W_t(|f|^2)+ |\tau| W_t(|g|^2)}{|\beta| + |\tau|} \ \ \ \ \text{from \eqref{nr1} and \eqref{nr2}.}\\
 & < \frac{|\beta| W_t(|g|^2)+ |\tau| W_t(|g|^2)}{|\beta| + |\tau|} \ \ \ \ \text{since} \ \ W_t(|f|^2) < W_t(|g|^2)\\
 & = W_t(|g|^2) 
\end{split} 
\end{equation}
Again
\begin{equation} \label{nr5}
\begin{split}
W_t(f\bar{g}) & = \frac{|\beta| W_t(|f|^2)+ |\tau| W_t(|g|^2)}{|\beta| + |\tau|}\\
 & > W_t(|f|^2)
\end{split} 
\end{equation}
since $W_t(|f|^2) < W_t(|g|^2)$. Thus from \eqref{nr4} and \eqref{nr5} we obtain $$ W_t(|f|^2) < W_t(f\bar{g}) < W_t(|g|^2).$$

\item[C.]i. Suppose that $W_t(|f|^2)$ cannot be computed. Then clearly $f \notin \Gamma_1.$ We know that $f = f_1f_2$ and that $|\alpha| > 0$ and $|\beta| \geq 0$. If $W_t(|f|^2)$ cannot be computed, then we get that $\sum_{i=1}^c (\alpha_i + \hat{\alpha}_i)\mu_i \geq 1.$ Thus $\sum_{i=1}^c \alpha_i \mu_i \geq 1/2$ since $\alpha_i = \hat{\alpha}_i$ for all $i, \ 1 \leq i \leq c.$

Let $g = g_1g_2$, where $g_1 \in \Gamma_3$ and $g_2 \in \Gamma_1.$ $W_t(|g|^2)$ and $W_t(f\bar{g})$ can be computed, and so $g$ cannot belong to $\Gamma_3,$ which implies the multiindex corresponding to $g_2$ must satisfy $|\tau| >0$. Therefore,
\begin{equation} \label{nr6}
\begin{split}
W_t(f\bar{g}) & = \frac{\frac{1}{2} - \sum_{i=1}^{c}\alpha_i \mu_i + \frac{1}{2} - \sum_{i=1}^c \gamma_i\mu_i}{|\beta| + |\tau|}\\
 & = \frac{\frac{1}{2} - \sum_{i=1}^{c}\alpha_i \mu_i + |\tau| \, W_t(|g|^2)}{|\beta| + |\tau|} \ \quad \text{from \eqref{nr2}}\\
 & \leq \frac{|\tau|}{|\beta| + |\tau|} W_t(|g|^2) \ \quad \text{since} \ \frac{1}{2} - \sum_{i=1}^c \alpha_i \mu_i \leq 0\\
 & \leq W_t(|g|^2) \ \quad \text{since} \ \frac{|\tau|}{|\beta| + |\tau|} \leq 1.
\end{split} 
\end{equation}

\item[ii.] From (C.i.) we know that $f = f_1f_2 \notin \Gamma_1$ and also that $\sum_{i=1}^c \alpha_i \mu_i \geq1/2$ since $\alpha_i = \hat{\alpha}_i$ for all $i, \ 1 \leq i \leq c.$ Similarly, for $g = g_1g_2$ it means that $g \notin \Gamma_1$ and that $\sum_{i=1}^c \gamma_i \mu_i \geq1/2$. Thus $\sum_{i=1}^c (\alpha_i + \gamma_i)\mu_i \geq1$, and so the pair of multiindices corresponding to the cross term $f\bar{g}$ is not in the set $\Theta_t$. Hence the number $W_t(f\bar{g})$ cannot be computed.

\item[D.]i. Suppose that $f$ is such that $|f|^2$ is in the leading polynomial $P_t$. Then $f \in \Gamma_3,$ which implies that $f = f_1f_2$ with $|\beta| =0$, that is, $f = C f_1$ for $C$ a non-zero constant. Let $g = g_1g_2$ be any monomial such that $W_t(|g|^2)$ can be computed. Clearly, $g \notin \Gamma_3,$ and so $|\tau| >0$ whereas $|\gamma| \geq 0.$ We know from \eqref{nr2} that 
\[W_t(|g|^2) = \frac{\frac{1}{2} - \sum_{i=1}^c\gamma_i\mu_i}{|\tau|}\]
Given that $|f|^2$ is a term in $P_t$, $\sum_{i=1}^c (\alpha_i+ \hat{\alpha}_i)\mu_i = 1$ since $P_t$ is a $\Lambda_t$-homogeneous polynomial of weighted degree 1. Thus $\sum_{i=1}^c \alpha_i\mu_i = 1/2$ since $\alpha_i = \hat{\alpha}_i$ for all $i, \ 1 \leq i \leq c.$
Now 
\begin{equation} \label{nr7}
\begin{split}
W_t(f\bar{g}) & = \frac{\frac{1}{2} - \sum_{i=1}^{c}\alpha_i \mu_i + \frac{1}{2} - \sum_{i=1}^c \gamma_i\mu_i}{|\beta| + |\tau|}\\
 & = \frac{\frac{1}{2} - \sum_{i=1}^{c}\alpha_i \mu_i + \frac{1}{2} - \sum_{i=1}^c \gamma_i\mu_i}{|\tau|} \ \quad \text{since} \ |\beta| =0\\
 & = \frac{\frac{1}{2} - \sum_{i=1}^{c}\gamma_i \mu_i}{|\tau|} \ \quad \text{since} \ \sum_{i=1}^c \alpha_i \mu_i= \frac{1}{2}\\
 & = W_t(|g|^2)
\end{split} 
\end{equation}

\item[ii.] Let $g = g_1g_2$ be any monomial such that $W_t(|g|^2)$ cannot be computed. From (D.i.) we know that $\sum_{i=1}^c \alpha_i\mu_i = 1/2$ and from (C.ii.) we know that $\sum_{i=1}^c \gamma_i\mu_i \geq 1/2$. Thus 
\[\sum_{i=1}^c (\alpha_i + \hat{\gamma}_i)\mu_i = \sum_{i=1}^c (\alpha_i + \gamma_i)\mu_i = \frac{1}{2} + \sum_{i=1}^c \gamma_i \mu_i \geq 1. \]
Hence the pair of multiindices corresponding to the cross term $f\bar{g}$ does not belong to the set $\Theta_t$, and so the number $W_t(f\bar{g})$ cannot be computed.
\end{itemize}

\end{proof}

\begin{remark}
The terms that are added to the leading polynomial $P_t$ at step $t$ of the Kol\' a\v r algorithm are the terms needed to compute the multitype. Clearly, the maximum $W_t$ at every step is always assumed by a square of a monomial. All the cross terms that are added to the leading polynomial could be eliminated without changing the multitype. Most importantly, the multitype of $M$ at the origin will not change if:
\begin{itemize}
 \item[i.] All the cross terms are eliminated from the expansion of $|h|^2,$ where $h$ is a generator of a sum of squares domain and
 \item[ii.] All the squares that do not appear in the final leading polynomial are added to our data.
 \end{itemize}
\end{remark}

\section{Proofs of the Main Results}
We give the proofs of the two main theorems in this section. As hinted in the introduction, we shall use the result from Lemma \ref{nl} to prove Theorems \ref{thm1} and \ref{thm2}. By our convention in this paper, we assume without loss of generality that the domains we work with are decoupled as mentioned in the introduction. We start by proving a more general result than Theorem \ref{thm1}.

\begin{theorem} \label{thm5}
Let $0 \in M \subset \mathbb{C}^{n+1}$ be a hypersurface whose defining function is given as
\[r(z) = \text{2Re}(z_{n+1}) + \sum_{j=1}^{N}|f_j(z_1, \dots, z_n)|^2,\]
where $f_1, \dots, f_N$ are holomorphic functions near the origin and assume that the D'Angelo 1-type of $M$ at the origin is finite. Then the leading polynomial $P_t(z_1,\dots, z_{n-d_t},\bar{z}_1,\dots, \bar{z}_{n-d_t})$ obtained at step $t$ of the Kol\' a\v r algorithm, for $t \geq 1,$ is a sum of squares of holomorphic polynomials. In particular, the final leading polynomial $P(z_1,\dots, z_n,\bar{z}_1,\dots, \bar{z}_n)$ corresponding to the multitype weight is likewise a sum of squares.
\end{theorem}

\begin{proof}
We order the generators $f_k$ by vanishing order. We truncate each generator $f_k$ up to order $\beta=\lceil \Delta_1(M,0) \rceil,$ the ceiling of the D'Angelo 1-type of $M$ at the origin, and denote it by $f_k^{\beta}.$ Since a sum of squares domain is pseudoconvex, by Theorem~\ref{catlinthm}, no terms of higher order than $\beta$ ought to come into the computation of the multitype. We order the terms in the truncated generator $f_k^{\beta}$ by vanishing order and also use the reverse lexicographic ordering to reorder the monomials with the same combined degree. Now let $f_{k,i} = C_{k,i} z_{1}^{{\alpha}_{1}^{k,i}} \cdots z_{n}^{{\alpha}_{n}^{k,i}}$ be the $i$-th monomial in the generator $f_k^{\beta}$ after ordering by vanishing order.  Let the number of distinct combined degrees in the Taylor expansion $f_k^{\beta}$ be $\kappa_k$, and let $\eta_{k,j}$ be the number of non-zero monomials with the same combined degree $\nu_{k,j}$ in $f_k^{\beta}$. Thus,
\[|f_k^{\beta}|^2 = |f_{k,1} + \cdots + f_{k,\eta_k}|^2, \:\: \text{where} \:\: \eta_k = \sum_{j=1}^{\kappa_k} \eta_{k,j}. \]
Here $\eta_k$ is the total number of monomials with nonzero coefficients in the power series expansion of the generator $f_k$ up to order $\beta.$ In the expansion of $|f_k^{\beta}|^2,$ we have two types of terms: squares $|f_{k,i}|^2$ and cross terms $\text{2Re}(f_{k,j}\bar{f}_{k,i})$. For simplicity sake, for each monomial $f_{k,i}$ in the generator $f_k^{\beta},$ we write the terms from the expansion of $|f_k^{\beta}|^2$ into an expression of the form 
\begin{equation} \label{h0}
|f_{k,i}|^2 + \sum_{j=1}^{i-1} \  \text{2Re} (f_{k,j}\bar{f}_{k,i}), 
\end{equation} 
for $i=1,2, \dots, \eta_k.$ 

Define $P_{t,k}$ for $t \geq 1$ to be the sum in the leading polynomial $P_t$ at step $t$ consisting of terms from the expansion of $|f_k^{\beta}|^2$. We could have $P_{t,k}=0$ for some $k, \ 1 \leq k \leq N,$ and for some $t \geq 1$ if there are no terms from the expansion of $|f_k^{\beta}|^2$ in the leading polynomial $P_t$ after the $t$-th step. Thus, $P_t = \sum_{k=1}^N P_{t,k}$. In order to show that the final leading polynomial $P_t$ is a sum of squares, it will suffice to show that each $P_{t,k}$ obtained at each step of the Kol\' a\v r algorithm is a sum of squares. Note that trivially $0$ is a sum of squares.

The Bloom-Graham type is $2\nu_1$, where $\nu_1: = \displaystyle \min_{1 \leq k \leq N} \{\nu_{k,1}\}=\nu_{1,1}$ since the monomials from the expansion of each $|f_k^{\beta}|^2$ as well as the generators $f_k$ are ordered by vanishing order. Thus, the leading polynomial $P_1$ consists of all terms of combined degree $2\nu_1$. Clearly, $P_{1,1}\neq 0$, but $P_{1,k}=0$ for every $k$ such that $\nu_{k,1} > \nu_1.$ Let $k$ be such that $P_{1,k}\neq 0,$ and denote by $m_k \geq 1$ the number of monomials from the Taylor expansion of $f_k^{\beta}$ that have combined degree $\nu_1.$ By writing the terms from the expansion of $|f_k^{\beta}|^2$ into the form given in \eqref{h0}, it is easy to see that the first $m_k$ squares $|f_{k,i}|^2$ for $i= 1, \dots, m_k$ as well as the 
$\begin{pmatrix}
m_k \\ 2 \
\end{pmatrix}$
cross terms $\text{2Re}(f_{k,j}\bar{f}_{k,i})$ for $j<i$ all have combined degree $2\nu_1$. Thus
\begin{equation} \label{h1}
P_{1,k} = |f_{k,1} + \cdots + f_{k,m_k}|^2,
\end{equation}
which is obviously a sum of squares. If $d_1 = 0$, then we are done and $P_1$ becomes the leading polynomial with multitype weight $\Lambda_1$. On the other hand, if $d_1 = n-c$ for $c <n$, then we proceed to the next step.

In the second step, we assume without loss of generality that the $W_1$ value of at least one square in $\text{Q}_1$ can be computed. Then the maximum $W_1$ value exists, which further implies that some terms from the expansion of $|f_k^{\beta}|^2$ for some $k$ will be added to the leading polynomial $P_1$ in order to obtain $P_2$. $P_{2,k}$ might be $0$ for some $k$ if no terms from the expansion of $|f_k^{\beta}|^2$ end up in $P_2$ after this step. Obviously, the interesting case is when $P_{2,k}\neq 0$. Consider $k$ such that $P_{2,k}\neq 0$. Suppose that $u_k$ squares from the expansion of $|f_k^{\beta}|^2$ give the maximum $W_1$ value, and let $|f_{k,c_l}|^2$ for $l=1, \dots, u_k$ be these squares. Here $c_l$ for $l=1, \dots , u_k$ is some positive integer between $1$ and $\eta_k$, and $c_l < c_{l+1}$ for $l=1, \dots, u_k-1$. The argument now splits into two cases:
\begin{itemize}
\item[CASE {\Large 1}:] $P_{1,k}=0$. By Lemma~\ref{nl} part A, we get exactly  
$\begin{pmatrix}
u_k \\ 2 \
\end{pmatrix}$
cross terms $\text{2Re}(f_{k,c_e}\bar{f}_{k,c_l})$, which give the maximal value for $W_1$ and $c_e < c_l$ for $1 \leq e,l \leq u_k$. Combining the $u_k$ squares with the 
$\begin{pmatrix}
u_k \\ 2 \
\end{pmatrix}$
cross terms gives 
\begin{equation} \label{h2}
P_{2,k}= |f_{k,c_1} + \cdots + f_{k,c_{u_k}}|^2.
\end{equation}
\item[CASE {\Large 2}:] $P_{1,k}\neq 0$. Then from Lemma~\ref{nl} part Di, we know that for each square $|f_{k,c_l}|^2$ there are exactly $m_k$ cross terms $\text{2Re}(f_{k,j}\bar{f}_{k,c_l})$ for all $j < c_l \ \text{and} \ j = 1, \dots, m_k$ as well as $l= 1, \dots, u_k$ that give the maximal value for $W_1$. We also know from the first case that there are 
$\begin{pmatrix}
u_k \\ 2 \
\end{pmatrix}$
cross terms $\text{2Re}(f_{k,c_e}\bar{f}_{k,c_l})$ that give the maximal value for $W_1$, and so we obtain the result that 
\begin{equation} \label{h3}
P_{2,k}= |f_{k,1} + \cdots + f_{k,m_k} + f_{k,c_1} + \cdots + f_{k,c_{u_k}}|^2.
\end{equation}
In the $(t+1)$-th step, we begin by first assuming that the sum $P_{t,k}$ for some $k$ is a sum of squares because if $P_{t,k}=0$ the argument is identical to the one given in Case 1. Thus let $P_{t,k}$ be given as $P_{t,k} = |f_{k,b_1} + \cdots + f_{k,b_{v_k}}|^2$, where $v_k$ is the total number of squares from the expansion of $|f_k^{\beta}|^2$ in $P_t$ after step $t$ with $b_j < b_{j+1}$ for $j=1, \dots, v_k-1$. Let's assume that the $W_t$ value of at least one square in $Q_t$ can be computed. This implies that the maximal value for $W_t$ exists.  
Assume that $s_k$ squares from the expansion of $|f_k^{\beta}|^2$ give the maximum $W_t$ value, and let $|f_{k,a_l}|^2$ for $l=1, \dots, s_k$ be these squares. 

From Lemma \ref{nl} part Di, we know that for each square $|f_{k,a_l}|^2$ there are exactly $v_k$ cross terms $\text{2Re}(f_{k,j}\bar{f}_{k,a_l})$ for all $j< a_l$ where $j = b_1, \dots, b_{v_k}$ and $l= 1, \dots, s_k$, which give the maximal value for $W_t$. By Lemma~\ref{nl} part A, there are 
$\begin{pmatrix}
s_k \\ 2 \
\end{pmatrix}$
cross terms $\text{2Re}(f_{k,a_e}\bar{f}_{k,a_l})$ which give the maximal value for $W_t$, and so we obtain the result that 
\begin{equation} \label{h4}
P_{t+1,k}= |f_{k,b_1} + \cdots + f_{k,b_{v_k}} + f_{k,a_1} + \cdots + f_{k,a_{s_k}}|^2.
\end{equation}
Hence $P_{t+1,k}$ is a sum of squares. Since the leading polynomial $P_t$ at each step is a sum of the $P_{t,k}$'s, the leading polynomial at each step and subsequently the final leading polynomial is a sum of squares.
\end{itemize}
\medskip  Since every change of variables that is allowed in the Kol\' a\v r algorithm sends each square to a square and keeps the weight of terms in the leading polynomial $P_t$ the same, it is easy to see that the leading polynomial is still a sum of squares after any change of variables.

\medskip Our assumption of finite D'Angelo 1-type implies that the last entry of the multitype is bounded, and so all entries of the multitype weight will be finite. This means that the Kol\' a\v r algorithm will definitely terminate after a finite number of steps, and so the above procedure can only occur finitely many times. We conclude therefore that the final leading polynomial corresponding to the last weight in the procedure, the multitype weight, is always a sum of squares, too.
\end{proof}

\noindent{\textbf{Proof of Theorem \ref{thm1}}}: We know that the final leading polynomial in the Kol\' a\v r algorithm for computing the multitype at the origin is the model polynomial. Thus, from Theorem \ref{thm5}, we conclude that the model of a sum of squares is a sum of squares.
\begin{flushright}
$\Box$
\end{flushright} 

\vspace{0.5cm}

Before we proof Theorem \ref{thm2} we shall consider the following lemma:

\begin{lemma}\label{mulidinv}
Let $ M \subset \mathbb{C}^{n+1}$ with $0 \in M$  be a hypersurface whose defining function is given by
\[r(z) = \text{2Re}(z_{n+1}) + \sum_{j=1}^{N}|f_j(z_1, \dots, z_n)|^2,\]
where $f_1, \dots, f_N$ are holomorphic functions near the origin.  Let $M' \subset \mathbb{C}^{n+1}$ be another hypersurface whose defining function is given as
\[u(z) = \text{2Re}(z_{n+1}) + \sum_{j=1}^{l-1}|f_j|^2 + |h_lf_l - \sum_{c \neq l} h_c f_c|^2 + \sum_{j=l+1}^{N}|f_j|^2,\]
for some fixed $l$, where $h_j$ is a holomorphic function near the origin for every $j=1,\dots,N.$ Assume that the D'Angelo 1-type of $M$ is finite at the origin. Then the multitype obtained by applying Kol\' a\v r algorithm to both $r(z)$ and $u(z)$ is the same provided that $\left <f_1, \cdots, f_N \right >$ and $$\left <f_1, \cdots,f_{l-1},h_lf_l - \sum_{c \neq l} h_c f_c,f_{l+1},\cdots  f_N \right >$$ represent the same ideal in the ring $\mathcal{O}.$
\end{lemma}

\begin{remark}
Modifying only one generator at a time makes the bookkeeping in the computation of the multitype easier to follow.
\end{remark}

\begin{remark}\label{lincomb}
Assume that there exist some $l$ such that $1 \leq l \leq N$ and holomorphic functions near the origin $h_1, \dots, h_{l-1}, h_{l+1}, \dots, h_N$ such that $\displaystyle f_l =  \sum_{c \neq l} h_c f_c.$ It is clear that $\left <f_1, \cdots, f_N \right >$ and $\left <f_1, \cdots,f_{l-1},f_{l+1},\cdots  f_N \right >$ represent the same ideal in $\mathcal{O}.$ Therefore, applying Lemma~\ref{mulidinv} with $h_l \equiv 1$ shows that adding in the square $|f_l|^2 $ or taking it away makes absolutely no difference as far as the multitype computation goes. This observation will be crucial in the proof of Theorem \ref{thm2}.
\end{remark}

\begin{proof}
We will show that the multitype obtained by applying the Kol\' a\v r algorithm to both $r(z)$ and $u(z)$ is the same. Let $\beta=\lceil \Delta_1(M,0) \rceil$ be the ceiling of the D'Angelo 1-type of $M$ at the origin. We truncate each generator $f_k$ as well as each holomorphic function $h_k$ at the order $\beta$ and denote them by $f_k^{\beta}$ and $h_k^{\beta}$ respectively. Denote by $f_{k,i}$ the $i$-th monomial from the Taylor expansion of $f_k^{\beta}$ after ordering by vanishing order and reverse lexicographic order for the monomials with the same vanishing order. 

By Lemma \ref{nl}, we know that each entry of the multitype is realized by a square. If no square from the expansion of $|f_l^{\beta}|^2$ contributes to the entries of the multitype, then the multitype entries for both defining functions $r(z)$ and $u(z)$ are the same, and there is nothing to prove. We thus assume that there exists at least one square from the expansion of $|f_l^{\beta}|^2$ that contributes to the entries of the multitype and that no nonzero multiple of that square exists in any of the expansions of $|f_1^{\beta}|^2, \dots, |f_{l-1}^{\beta}|^2, |f_{l+1}^{\beta}|^2, \dots, |f_N^{\beta}|^2.$

Next, we claim that if $h_l (0)=0,$ then no square from the expansion of $|h_l^{\beta}f_l^{\beta}|^2$ can contribute to the entries of the multitype. Indeed, let $h_{l,i}$ be the $i$-th monomial from the Taylor expansion of $h_l^{\beta}$ after ordering by vanishing order and reverse lexicographic order for the monomials with the same vanishing order. For every monomial $f_{l,j}$ in $ f_l^{\beta},$ the monomial $h_{l,i}f_{l,s}$ in $ h_l^{\beta}f_l^{\beta}$ has greater combined degree than that of $f_{l,j}$ for every $i\geq 1.$ As a result, $|h_{l,i}f_{l,j}|^2$ cannot give the Bloom-Graham type since the combined degree of $|f_{l,j}|^2$ is strictly less than the combined degree of $|h_{l,i}f_{l,j}|^2$. Furthermore, $W_t(|h_{l,i}f_{l,j}|^2)$ cannot be computed if $W_t(|f_{l,j}|^2)$ gives the maximum $W_t$-value. By the same argument, if  $f_l=\sum_{c \neq l} g_c f_c,$ for $g_c$ with $c=1,\dots, l-1, l+1, \dots, N$ holomorphic functions near the origin, then no square from the expansion of $|f_l^{\beta}|^2$ can contribute to the entries of the multitype unless it is the square of the nonzero constant term of some $g_c$ multiplied by a monomial of $f_c$ that in itself gives that same multitype entry. Since we assumed the contrary, $f_l$ cannot be written in terms of the other generators. By our hypothesis, however, $\left <f_1, \cdots, f_N \right >$ and $\left <f_1, \cdots,f_{l-1},h_lf_l - \sum_{c \neq l} f_ch_c,f_{l+1},\cdots  f_N \right >$ represent the same ideal in the ring $\mathcal{O}.$ Putting these two facts together along with our assumption  that there exists at least one square from the expansion of $|f_l^{\beta}|^2$ that contributes to the entries of the multitype and that no nonzero multiple of that square exists in any of the expansions of $|f_1^{\beta}|^2, \dots, |f_{l-1}^{\beta}|^2, |f_{l+1}^{\beta}|^2, \dots, |f_N^{\beta}|^2,$ we conclude that $h_l (0) \neq 0.$ Without loss of generality, assume $h_l \equiv 1.$ We shall show that modifying the function $f_l^{\beta}$ in the sum of squares by the sum $\sum_{c \neq l} f_c^{\beta}h_c^{\beta}$ does not alter the multitype. We further assume that the sum $\sum_{c \neq l} f_c^{\beta}h_c^{\beta} \neq 0.$ By Lemma \ref{nl}, it suffices to focus on how the omission of some squares of monomials from the term $|f_l^{\beta} - \sum_{c \neq l}h_c^{\beta}f_c^{\beta}|^2$ in the defining function $u(z)$ affects our results. We now break our argument into two cases:

\begin{itemize}
\item[CASE \large 1:] Assume that there exists a monomial $m$ in $f_l^{\beta}$ whose square $|m|^2$ from the expansion of $|f_l^{\beta}|^2$ gives the Bloom-Graham type. We consider two subcases here:
\item[i.] Assume that $m$ is in the expression $f_l^{\beta} - \sum_{c \neq l}h_c^{\beta}f_c^{\beta}$. Then no monomial in the sum $\sum_{c \neq l}h_c^{\beta}f_c^{\beta}$ cancels out $m.$ Clearly, the weights obtained at the first step of the Kol\' a\v r algorithm are the same for both $r(z)$ and $u(z)$ since $|m|^2$ belongs to both defining functions.

\item[ii.] Assume that $m$ is not in the expression $f_l^{\beta} - \sum_{c \neq l}h_c^{\beta}f_c^{\beta}$. Hence, $m$ gets cancelled out in the expression $f_l^{\beta} - \sum_{c \neq l}h_c^{\beta}f_c^{\beta}$ and so does not appear in $u(z).$ Let $\psi$ be the monomial in the sum $\sum_{c \neq l} h_c^{\beta}f_c^{\beta}$ that cancels out $m,$ and write $\psi = h_{c,i}f_{c,j}$ for some $c \neq l,$ where $h_{c,i}$ is some monomial in $h_c^{\beta}$ and $f_{c,j}$ is some monomial in $f_c^{\beta}$. By our assumption, $\psi$ equals $m,$ and its square $|\psi|^2$ gives the Bloom-Graham type as well. The monomial $h_{c,i}$ cannot have vanishing order 1 or higher; otherwise, $f_{c,j}$ must have combined degree less than that of $\psi,$ which contradicts the fact that $|\psi|^2$ gives the Bloom-Graham type. Thus, $h_{c,i} =h_{c,1} \in \C$ and $m = h_{c,1} f_{c,j}.$ Hence, the square $|f_{c,j}|^2$ gives the Bloom-Graham type as well. Even though there is the cancellation in $u(z),$ the weight obtained at the first step having applied the Kol\' a\v r algorithm to $r(z)$ and $u(z)$ is the same. More specifically, the squares $|f_{c,j}|^2$ and $|m|^2$ appear in the expansions of $|f_c^{\beta}|^2$ and $|f_l^{\beta}|^2$ respectively.

\item[CASE \large 2:] Assume that there exists a monomial $m$ in $f_l^{\beta}$ whose square $|m|^2$ from the expansion of $|f_l^{\beta}|^2$ gives the maximum $W_t$-value at the $(t+1)$-th step for $t \geq 1.$ 
\item[i.] Assume that $m$ is in the expression $f_l^{\beta} - \sum_{c \neq l}h_c^{\beta}f_c^{\beta}$. Then no monomial in the sum $\sum_{c \neq l}h_c^{\beta}f_c^{\beta}$ cancels out $m,$ and so the weights obtained at the $(t+1)$-th step are the same for both $r(z)$ and $u(z)$ since $|m|^2$ belongs to both defining functions.

\item[ii.] Assume that $m$ is not in the expression $f_l^{\beta} - \sum_{c \neq l}h_c^{\beta}f_c^{\beta}$. Then $m$ gets cancelled out by some monomial $\psi$ in the sum $\sum_{c \neq l} f_c^{\beta}h_c^{\beta}.$ Let $\psi = h_{c,s}f_{c,j}$ for some $s$ and $j,$  where $h_{c,s}$ is some monomial in $h_c^{\beta}$ and $f_{c,j}$ is some monomial in $f_c^{\beta}$. This implies that $m = \psi$ and $|\psi|^2$ gives the maximal $W_t$-value at step $t+1$ as well. Now let's assume that $h_{c,s} \notin \C.$ Then the combined degree of $f_{c,j}$ is less than that of $\psi$. 

If $W_t(|f_{c,j}|^2)$ cannot be computed, then $W_t(|\psi|^2)$ cannot be computed, which gives a contradiction. Therefore, we assume that $W_t(|f_{c,j}|^2)$ can be computed.
\end{itemize}

\noindent At this point let us recall the definition of $\Gamma_1,$ $\Gamma_2,$ and $\Gamma_3$ as given in the proof of Lemma \ref{nl}. Let $\Gamma_1$ be the set of all non-zero monomials that consist of only variables not in $\text{P}_t$, let $\Gamma_2$ be the set of all non-zero monomials which consist of variables both in $P_t$ as well as variables not in $P_t$, and let $\Gamma_3$ be the set of all non-zero monomials which consist of only variables in $P_t$. Also, recall that for any monomial $f$ if $W_t(|f|^2)$ can be computed, then $f \in \Gamma_1$ or $f \in \Gamma_2$ only. Again, we shall write any monomial $f$ in the form $f = \gamma_1\, \gamma_2$ where $\gamma_1$ and $\gamma_2$ are monomials satisfying $\gamma_1 \in \Gamma_3$ and $\gamma_2 \in \Gamma_1.$ Recall that $$W_t(|f|^2) = \frac{1 - \sum_{i=1}^{\kappa}(\alpha_i+\hat{\alpha}_i)\mu_i}{\sum_{i=\kappa +1}^n(\alpha_i+\hat{\alpha}_i)},$$ where $(\alpha_1, \dots, \alpha_n, \hat{\alpha}_1,\dots, \hat{\alpha}_n)$ is the multiindex of the monomial $|f|^2$ whose $W_t$ is being computed, $\kappa$ is the number of variables in the leading polynomial $P_t$, and $W_t(|f|^2)$ is the $W_t$-value of the term $|f|^2$. 

\medskip
We shall now consider the number $W_t(|\psi|^2)$ given that $W_t(|f_{c,j}|^2)$ can be computed. From Lemma \ref{nl}, $f_{c,j} \in \Gamma_1$ or $\Gamma_2.$ Since the monomial $h_{c,s}$ can belong to $\Gamma_1,$ $\Gamma_2,$ or $\Gamma_3,$ we shall consider three subcases below and assume that $f_{c,j} \in \Gamma_1$ or $\Gamma_2$ in each case: 

\begin{itemize}
\item[a.] Assume that $h_{c,s} \in \Gamma_1.$ Clearly, $W_t(|f_{c,j}|^2)$ and $W_t(|\psi|^2)$ both have the same numerator and the denominator of $W_t(|\psi|^2)$ is greater than that of $W_t(|f_{c,j}|^2)$ because $h_{c,s} \in \Gamma_1.$ Thus, $W_t(|f_{c,j}|^2)$ is greater than $W_t(|\psi|^2),$ which is a contradiction to our hypothesis that $W_t(|\psi|^2)$ is maximal at step $t+1.$

\item[b.] Assume that $h_{c,s} \in \Gamma_2.$ Here, the numerator of $W_t(|\psi|^2)$ is smaller than the numerator of $W_t(|f_{c,j}|^2)$ since $h_{c,s}$ contains a monomial in $\Gamma_3.$ Also, the denominator of $W_t(|\psi|^2)$ is greater than the denominator of $W_t(|f_{c,j}|^2)$ because $h_{c,s}$ contains a monomial in $\Gamma_1.$ Thus, $W_t(|f_{c,j}|^2)$ is greater than $W_t(|\psi|^2)$, which is again a contradiction.

\item[c.] Assume that $h_{c,s} \in \Gamma_3.$ Then $W_t(|f_{c,j}|^2)$ is always greater than $W_t(|\psi|^2)$ for $f_{c,j} \in \Gamma_1$ or $\Gamma_2$ since the denominators of both numbers are equal and the numerator of $W_t(|\psi|^2)$ is less than that of $W_t(|f_{c,j}|^2).$ This gives a contradiction since $W_t(|\psi|^2)$ is maximal at step $t+1.$ 

\medskip From cases (a), (b), and (c) we can see that if $h_{c,s} \notin \C$, then $W_t(|\psi|^2)$ cannot be the maximum at step $t+1,$ and so we have a contradiction to our hypothesis in all three cases. Hence $h_{c,s} \in \C$ and so $m = h_{c,s} f_{c,j}.$ Clearly, $W_t(|f_{c,j}|^2)$ gives the maximal value at step $t+1,$ too. This implies that if we apply the Kol\' a\v r algorithm to both $r(z)$ and $u(z),$ then the multitype entry at the $(t+1)$-th step will be the same for both defining functions. The squares $|f_{c,j}|^2$ and $|m|^2$ appear in the expansions of $|f_c^{\beta}|^2$ and $|f_l^{\beta}|^2$ respectively. We see that regardless of the cancellation in $u(z),$ the weight obtained at the $(t+1)$-th step remains unchanged.   

\end{itemize}

\noindent Clearly, the case when $h_l(0)\neq 0$ combines the analysis for the cases when $h_l \equiv 1$ and $h_l (0)=0.$
\end{proof}

\noindent \textbf{Proof of Theorem \ref{thm2}}: 
Let $M \subset \mathbb{C}^{n+1}$ with $0 \in M$ be a hypersurface whose defining function is given by
\begin{equation} \label{mm}
r(z) = \text{2Re}(z_{n+1}) + \sum_{j=1}^{N}|f_j(z_1, \dots, z_n)|^2,
\end{equation}
where $f_1, \dots, f_N$ are holomorphic functions near the origin. We begin by letting $M' \subset \mathbb{C}^{n+1}$ be another hypersurface whose defining function is given as
\begin{equation} \label{mm1}
u(z) = \text{2Re}(z_{n+1}) + \sum_{j=1}^{S}|g_j(z_1, \dots, z_n)|^2,
\end{equation} 
where $g_1, \dots, g_S$ are also holomorphic functions near the origin, and show that the multitype obtained by applying the Kol\' a\v r algorithm to both $r(z)$ and $u(z)$ is the same provided that $\left <f_1, \cdots, f_N \right >$ and $\left <g_1, \cdots,g_S \right >$ represent the same ideal in the ring $\mathcal{O}.$ 

Let the ideals associated to the hypersurfaces $M$ and $M'$ be given by $\left<f \right>= \left<f_1, \dots, f_N \right>$ and $\left<g \right> = \left<g_1, \dots, g_S \right>$ respectively, and suppose that $\left< f \right> = \left< g\right>.$ By Remark~\ref{lincomb} following the statement of Lemma~\ref{mulidinv}, we know that adding in the square of any element of the ideal $\left<f_1, \dots, f_N \right>$ does not modify the multitype because that element can be written in terms of the generators $f_1, \dots, f_N$ with coefficients in $\mathcal{O}.$ Since $\left<f_1, \dots, f_N \right>=\left<g_1, \dots, g_S \right>,$ each $g_j$ is an element of $\left<f_1, \dots, f_N \right>$ and can be written in terms of 
$f_1, \dots, f_N$ with coefficients in $\mathcal{O}.$ Therefore, 
\eqref{mm} has the same multitype at the origin as \[r_1(z) = \text{2Re}(z_{n+1}) + \sum_{j=1}^{N}|f_j(z_1, \dots, z_n)|^2+|g_1(z_1, \dots, z_n)|^2,\] and inductively, the same multitype at the origin as  \[r_S(z) = \text{2Re}(z_{n+1}) + \sum_{j=1}^{N}|f_j(z_1, \dots, z_n)|^2+ \sum_{k=1}^{S}|g_k(z_1, \dots, z_n)|^2.\] Now, we apply the argument in reverse. Since $\left<g_1, \dots, g_S \right>=\left<f_1, \dots, f_N \right>,$ each $f_j$ is an element of $\left<g_1, \dots, g_S \right>$ and can be written in terms of $g_1, \dots, g_S$ with coefficients in $\mathcal{O}.$ Therefore, by Remark~\ref{lincomb}, 
\eqref{mm1} has the same multitype at the origin as \[u_1(z) = \text{2Re}(z_{n+1}) + \sum_{k=1}^{S}|g_k(z_1, \dots, z_n)|^2+|f_1(z_1, \dots, z_n)|^2,\] and inductively, as  \[r_S(z) = \text{2Re}(z_{n+1}) + \sum_{k=1}^{S}|g_k(z_1, \dots, z_n)|^2+ \sum_{j=1}^{N}|f_j(z_1, \dots, z_n)|^2.\] We conclude that $r(z)$ and $u(z)$ have the same multitype at the origin, namely that the multitype is an invariant of the ideal of generators. 


\begin{flushright}
$\Box$
\end{flushright} 

\vspace{0.5cm}

\section{An Ideal Restatement of the Kol\' a\v r Algorithm}
In this section, we give the ideal restatement of the Kol\' a\v r algorithm for the multitype computation for sum of squares domains in $\C^{n+1}.$ Before that, we shall consider the following:

Let $M \subset \C^{n+1}$ be a hypersurface given by $r =0$ with 
\[r = \text{2Re}(z_{n+1}) + \sum_{k=1}^{N}|f_k(z_1, \dots, z_n)|^2,\]
where $f_1, \dots, f_N$ are holomorphic functions defined on a neighborhood of the origin. Assume that the D'Angelo 1-type is finite at the origin and that  $\beta=\lceil \Delta_1(M,0) \rceil,$ the ceiling of the D'Angelo 1-type of $M$ at the origin. We truncate each holomorphic function $f_k$ at the order $\beta$ and let $f_{k,i}$ be the $i$-th monomial of the generator $f_k^{\beta}$ after ordering by vanishing order. The ideal corresponding to the defining function $r$ is given as $\mathcal{I}=(z_{n+1}, f_1^{\beta}, \dots, f_N^{\beta})$. We know that the term $z_{n+1}$ has weight 1. Following the original algorithm of Kol\' a\v r, we shall ignore the term $z_{n+1}$ and work with the corresponding ideal $\mathcal{I}=(f_1^{\beta}, \dots, f_N^{\beta})$. 

\medskip From Theorem \ref{thm1} we know that all leading polynomials produced are sums of squares. Therefore, any leading polynomial $P_j$ can be written in the form 
\begin{equation} \label{f1}
P_j = \sum_{k=1}^N \Big|\sum_{i = 1}^{v_k} f_{k,a_i} \Big|^2,
\end{equation}
where the $f_{k,a_i}$'s are the monomials from the generator $f_k^{\beta}$ of weighted degree $\frac{1}{2}$ with respect to $\Lambda_j$. We will associate to every leading polynomial $P_j$ the ideal $\mathcal{I}_{P_j}$ given by
\begin{equation} \label{f2}
\mathcal{I}_{P_j} = \Bigg(\sum_{i = 1}^{v_1} f_{1,a_i}, \dots , \sum_{i = 1}^{v_N} f_{N,a_i} \Bigg).
\end{equation}
It is convenient to introduce notation for each square in $P_j.$ Let $P_{j,k} = \Big|\sum_{i = 1}^{v_k} f_{k,a_i} \Big|^2$. Then its associated ideal $\mathcal{I}_{P_{j,k}}$ can be expressed as 
\begin{equation} \label{f3}
\mathcal{I}_{P_{j,k}} = \Bigg(\sum_{i = 1}^{v_k} f_{k,a_i} \Bigg) = (f_{k,a_1}+\cdots+ f_{k,a_{v_k}})
\end{equation}
Clearly,
\begin{equation} \label{f4}
\mathcal{I}_{P_j} = \sum_{k=1}^N \mathcal{I}_{P_{j,k}}.
\end{equation}
Recall also that each monomial in every leading polynomial is of weighted degree one with respect to the corresponding weight. As a result, the weighted degree of any monomial $f_{k,a_i}$ is exactly one half with respect to the corresponding weight. 

Thus, given $\mathcal{I}_{P_{j,k}} = (f_{k,a_1} +\cdots+ f_{k,a_{v_k}})$, the $f_{k,a_i}$'s are exactly the monomials from the generator $f_k^{\beta}$ of weighted degree $\frac{1}{2}$ with respect to $\Lambda_j$. 

\medskip Set the ideal $\mathcal{I} = \mathcal{I}_0.$ For $j \geq 1$, the ideals $\mathcal{I}_{P_j}, \ \mathcal{I}_j, \ \text{and} \ \mathcal{I}_{P_{j,k}}$ can be described as follows:

\begin{enumerate}
\item[1.] $\mathcal{I}_{P_j}$ is the ideal whose generators are precisely the terms from the generators of the ideal $I_{j-1}$ having weighted order exactly $\frac{1}{2}$ with respect to the weight $\Lambda_j.$ We refer to the ideal $\mathcal{I}_{P_j}$ as the \textit{leading polynomial ideal}. 
\item[2.] The ideal $\mathcal{I}_j$ is the ideal obtained after applying the chosen $\Lambda_j$-homogeneous transformation, which makes the generators of $\mathcal{I}_{P_j}$ to be independent of the largest number of variables, to $\mathcal{I}_{j-1}$. Simply put, $I_j$ is the ideal obtained after changing variables in the ideal $I_{j-1}$.
\item[3.] $\mathcal{I}_{P_{j,k}}$ is the principal ideal whose generator is the sum of monomials in the generator $f_k^{\beta}$ with weighted degree exactly $\frac{1}{2}$ with respect to the weight $\Lambda_j.$ The ideal $\mathcal{I}_{P_{j,k}}$ is the zero ideal if no monomial in the generator $f_k^{\beta}$ has weighted degree $\frac{1}{2}$ with respect to the weight $\Lambda_j.$
\end{enumerate} 

\medskip

\noindent{\textbf{The Kol\' a\v r Algorithm (Ideal Version):}} Set the ideal $\mathcal{I} = \mathcal{I}_0,$ and compute the vanishing order at the origin of $\mathcal{I}_0$, which is the same as the degree $\nu_1$ of the lowest order monomial in $\mathcal{I}_0$. We define the Bloom-Graham type as twice the vanishing order of $\mathcal{I}_0$. This gives the first entry of the multitype $m_1 $ and so let $m_1 = 1/\mu_1,$ where $\mu_1 = 2\nu_1$. Set the first weight to be $\Lambda_1 = (\mu_1, \dots, \mu_1)$. 

In the second step, consider all $\Lambda_1$-homogeneous transformations, and choose one that will make the set of all generators of the leading polynomial ideal $\mathcal{I}_{P_1}$ to be independent of the largest number of variables. Denote this number by $d_1$. In the local coordinates after such a $\Lambda_1$-homogeneous transformation, we obtain that $\mathcal{I}_{P_1}$ is the ideal whose generators consist of those monomials in the variables $z_1, \dots, z_{n-d_1}$, which are of weighted degree 1/2 with respect to  $\Lambda_1.$ Apply the chosen $\Lambda_1$-homogeneous transformation to the ideal $\mathcal{I}_0$ to obtain the ideal $\mathcal{I}_1$. The rest of the terms from the generators in $\mathcal{I}_1,$ which are not in $\mathcal{I}_{P_1},$ have weighted degrees strictly greater than $\frac{1}{2}$ with respect to $\Lambda_1$. 

We shall now give a slightly modified version of Kol\' a\v r's  $\Theta_1$ and $W_1$. If $\alpha^{k,j} = (\alpha_1^{k,j}, \dots, \alpha_n^{k,j})$ is the multiindex of a monomial $f_{k,j}$ from any generator of $\mathcal{I}_1$, which is not in $\mathcal{I}_{P_1},$ then $f_{k,j}$ is of the form 
\[f_{k,j} =  C_{k,j}^1 z^{\alpha^{k,j}} \ \ \text{and} \ \ \ |\alpha^{k,j}|_{\Lambda_1} > \frac{1}{2}.\] 
Let
$$\Theta_1 = \left\{\alpha^{k,j} \:\: \Bigg| \:\: \ C_{k,j}^1 \neq 0 \:\: \text{and} \:\: \sum_{i=1}^{n-d_1} \alpha_i^{k,j} \mu_i < \frac{1}{2} \right\}.$$
For every $\alpha^{k,j} \in \Theta_1$,

\begin{equation} \label{idw1}
W_1(\alpha^{k,j}) = \frac{\frac{1}{2} - \sum_{i=1}^{n-d_1} \alpha_i^{k,j} \mu_i}{\sum_{i=n-d_1+1}^{n} \alpha_i^{k,j}}.
\end{equation}
The next weight $\Lambda_2$ is defined by letting 

\[\lambda_i^2 = \max_{\alpha^{k,j} \in \Theta_1} W_1(\alpha^{k,j})\] 
for $i > n-d_1$, and $\lambda_i^2= \mu_1$ for $i \leq n-d_1$. To complete the second step, we let $\mathcal{I}_{P_2}$ be the second leading polynomial ideal corresponding to the weight $\Lambda_2$. The generators of $\mathcal{I}_{P_2}$ depend on more than $n-d_1$ variables.

We proceed by induction. At the step $t,$ for $t>2,$ we consider all $\Lambda_{t-1}$-homogeneous transformations and choose one that makes the generators of the leading polynomial ideal $\mathcal{I}_{P_{t-1}}$ to be independent of the largest number of variables. Denote this number by $d_{t-1}$. 
Apply this $\Lambda_{t-1}$-homogeneous transformation to the previous ideal $\mathcal{I}_{t-2}$ in the $(t-1)$-th step to obtain the ideal $\mathcal{I}_{t-1}$. We know from the Kol\' a\v r algorithm that the number of multitype entries that are added at each step of the computation depends on the difference $(d_{t-2} - d_{t-1})$. 
We consider two cases:

\begin{itemize}
\item[CASE {\large 1}:] Assume that $d_{t-2} > d_{t-1}$. Again recall that for any weight $\Lambda$ that is smaller than $\Lambda_{t-1}$ with respect to the lexicographic ordering, $\Lambda$-adapted coordinates are also $\Lambda_{t-1}$-adapted. This implies that we get $(d_{t-2} - d_{t-1})$ multitype entries
\[
\mu_{n-d_{t-2}+1} = \cdots = \mu_{n-d_{t-1}} = \lambda_{n- d_{t-2}+1}^{t-1}
\]
and let $\lambda_i^t = \mu_i$ for $i \leq n - d_{t-2}$. Here, $\mathcal{I}_{P_{t-1}}$ is the ideal whose generators are sums of monomials in  the variables $z_1, \dots, z_{n-d_{t-1}}$ that are $\Lambda_{t-1}$-homogeneous of weighted degree $\frac{1}{2}.$ To obtain $\lambda_i^t$ for $i>n-d_{t-1}$, we consider the rest of the monomials from the generators in $\mathcal{I}_{t-1}$ that are not in $\mathcal{I}_{P_{t-1}}$. Using these monomials that have weighted degree strictly greater than $\frac{1}{2}$ with respect to $\Lambda_{t-1}$, we define $\Theta_{t-1}$ and compute $W_{t-1}$ in a similar way as in step two. Such monomials are of the form $f_{k,j} =  C_{k,j}^{t-1} z^{\alpha^{k,j}}$ for multiindex  $\alpha^{k,j} = (\alpha^{k,j}_1, \dots, \alpha
^{k,j}_n)$ satisfying $|\alpha^{k,j}|_{\Lambda_{t-1}} > \frac{1}{2}.$ Thus,

$$\Theta_{t-1} = \left\{\alpha^{k,j} \:\: \Bigg| \:\: \ C_{k,j}^{t-1} \neq 0 \:\: \text{and} \:\: \sum_{i=1}^{n-d_{t-1}} \alpha_i^{k,j}  \mu_i < \frac{1}{2} \right\}.$$

For every $\alpha^{k,j} \in \Theta_{t-1}$,

\begin{equation} \label{idwr}
W_{t-1}(\alpha^{k,j}) = \frac{\frac{1}{2} - \sum_{i=1}^{n-d_{t-1}} \alpha_i^{k,j} \mu_i}{\sum_{i=n-d_{t-1}+1}^{n} \alpha_i^{k,j}}.
\end{equation}
So for the remaining multitype entries of $\Lambda_j,$ we let
\[
\lambda_i^t = \max_{\alpha^{k,j} \in \Theta_{t-1}} W_{t-1}(\alpha^{k,j}),\] 
for $i > n-d_{t-1}$. 

\item[CASE {\large 2}:] Assume that $d_{t-1} = d_{t-2}$. There are zero multitype entries computed in this case, and so we only determine $\lambda_i^t$ for $t > n-d_{t-1}$ using \eqref{idwr}. This completes the step $t$ of the algorithm. 
\end{itemize}

We can thus establish a one-to-one correspondence between the leading polynomial $P_t$ for $t \geq 1$ and the intermediate ideal $\mathcal{I}_{P_t}$ introduced above. Since working with ideals of holomorphic functions is often easier than with real-valued polynomials, the restatement of the Kol\' a\v r algorithm simplifies multitype computations for a sum of squares domain. We work with considerably fewer terms in the case of the ideals as compared to sums of squares. In particular, for each modulus square of a generator consisting of $m $ monomials, Kol\'a\v r's original algorithm involves working with $m$ squares plus $ \begin{pmatrix}
m \\ 2 \
\end{pmatrix}$ cross terms, whereas this restatement in terms of ideals involves computations for only $m$ monomials.

\vspace{0.5cm}
We give a corollary to Theorem \ref{thm1}:

\begin{corollary}\label{C1}
\textit{Let $M \subset \mathbb{C}^{n+1}$ be a hypersurface given by $r=0$ with 
\[r = \text{2Re}(z_{n+1}) + \sum_{j=1}^{N}|f_j(z_1, \dots, z_n)|^2,\]
where $f_1, \dots, f_N$ are holomorphic functions defined on a neighbourhood of the origin and assume that the D'Angelo 1-type of $M$ at the origin is finite. Then for $\ell \geq 1$, each monomial from every generator of the leading polynomial ideal \ $\mathcal{I}_{P_{\ell}}$ obtained at the $\ell$-th step of the Kol\' a\v r algorithm has weighted degree $\frac{1}{2}$ with respect to the weight $\Lambda_{\ell}$.}
\end{corollary}

\begin{proof}
 Let $f_k^{\beta}$ be the Taylor expansion of the holomorphic function $f_k$ to the order $\beta$, where $\beta$ is the ceiling of the D'Angelo 1-type. We order the generators by vanishing order and let $\mathcal{I} = (f_1^{\beta}, \dots, f_N^{\beta})$. Now assume that the vanishing order of the ideal $\mathcal{I}$ is $\nu >0.$ Then the Bloom-Graham type is precisely $2\nu$ and the weight $\mu_1 = \frac{1}{2\nu}$ with $\Lambda_1 = (\frac{1}{2\nu}, \dots, \frac{1}{2\nu}).$ Thus, $\mathcal{I}_{P_1}$ is not the zero ideal.

For every $k$ such that \ $\mathcal{I}_{P_{1,k}}$ is not the zero ideal,
\[\mathcal{I}_{P_{1,k}} = (f_{k,1}, \dots, f_{k,m_k}), \]
where each monomial $f_{k,i}$, for $1 \leq i \leq m_k,$ has weighted degree $\frac{1}{2}$ with respect to the weight $\Lambda_1.$ Next, assume that in the second step, the ideal $\mathcal{I}_{P_{1,k}}$ is the same as the ideal $\mathcal{I}_{P_{2,k}}$. We know that the entries corresponding to each variable in the monomial $f_{k,i},$ for $1 \leq i \leq m_k,$ are the same in both weights $\Lambda_1$ and $\Lambda_2.$ Therefore, each monomial from the generator $\sum_{i=1}^{m_k}f_{k,i}$ has weighted degree $\frac{1}{2}$ with respect to $\Lambda_2$ as well. Assume that the principal ideal $\mathcal{I}_{P_{2,k}}$ is generated by the sum $\sum_{i=1}^{m_k}f_{k,i} + \sum_{j=1}^{\gamma_k}f_{k,b_j}.$ Since the new sum $\sum_{j=1}^{\gamma_k}f_{k,b_j}$ corresponds to the new weight $\Lambda_2$, each monomial $f_{k,b_j}$ has weighted degree $\frac{1}{2}$ with respect to $\Lambda_2.$ Therefore, every monomial in the generator of the ideal $\mathcal{I}_{P_{2,k}}$ has weighted degree $\frac{1}{2}$ with respect to $\Lambda_2.$

Next, we assume that for $\ell \geq 2$
 \[ \mathcal{I}_{P_{\ell,k}} = (f_{k,a_1} + \cdots + f_{k,a_{v_k}}),\] 
where each monomial $f_{k,a_i}$, for $1 \leq i \leq v_k,$ has weighted degree exactly equal to $\frac{1}{2}$ with respect to the weight $\Lambda_{\ell}.$
 
Now, assume that at step $\ell+1$, the ideal $\mathcal{I}_{P_{{\ell+1},k}}$ is the same as the ideal $\mathcal{I}_{P_{\ell,k}}.$  Every monomial in the generator of  $\mathcal{I}_{P_{{\ell+1},k}}$ also has weighted degree equal to $\frac{1}{2}$ with respect to the weight $\Lambda_{\ell+1}$ because even though $\Lambda_{\ell}$ is not the same as $\Lambda_{\ell+1},$ the weight corresponding to each variable in $f_{k,a_{v_k}}$ is the same in both weights $\Lambda_{\ell}$ and $\Lambda_{\ell+1}$.

Next, assume that at step $\ell+1$ the  sum $\sum_{i=1}^{u_k}f_{k,b_{u_k}}$ is added to the sum $\sum_{i=1}^{v_k}f_{k,a_{v_k}}$ to obtain the generator of the ideal 
\[ \mathcal{I}_{P_{{\ell+1},k}} = (f_{k,a_1} + \cdots + f_{k,a_{v_k}}+ f_{k,b_1} + \cdots + f_{k,b_{u_k}}).\]
This implies that each monomial $f_{k,b_{u_k}}$ has weighted degree $\frac{1}{2}$ with respect to the weight $\Lambda_{\ell+1}$. Again, each monomial $f_{k,a_{v_k}}$ is of weighted degree $\frac{1}{2}$ with respect to the weight $\Lambda_{\ell+1}$ since the weight corresponding to each variable in $f_{k,a_{v_k}}$ is the same in both weights $\Lambda_{\ell}$ and $\Lambda_{\ell+1}$. Thus, every monomial from the generator of the ideal $\mathcal{I}_{P_{{\ell+1},k}}$ has weighted degree $\frac{1}{2}$ with respect to the weight $\Lambda_{\ell+1}.$
\end{proof}

\vspace{0.5cm}

\medskip\noindent The example that follows is the ideal restatement of the Kol\' a\v r algorithm applied to the defining function given in Example \ref{K1}.

\begin{example}\label{K3}
Let $M \subset \C^{4}$ be a sum of squares domain given by the defining function
\[r = \text{2Re}(z_4) + |z_1-z_2+z_3^2|^2 + |z_1^2-z_2^2|^2.\]
The associated ideal then becomes 
\[\mathcal{I}=(z_1 - z_2 + z_3^2, z_1^2-z_2^2)= \mathcal{I}_0.\]

\noindent The vanishing order here equals one, and so the Bloom-Graham type is 2, which implies that $\mu_1 = \frac{1}{2}$ and $\Lambda_1= (\frac{1}{2}, \frac{1}{2}, \frac{1}{2})$. Hence
\[\mathcal{I}_{P_1} = (z_1-z_2) = \mathcal{I}_{P_{1,1}}.\]
So $\mathcal{I}_{P_{1,2}}$ is the zero ideal. We consider all $\Lambda_1$-homogeneous transformations and choose $$\tilde{z}_1 = z_1 - z_2 \quad \tilde{z}_j = z_j, \quad j=2,3,4,$$ to make $\mathcal{I}_{P_1}$ independent of the largest number of variables. We shall ignore $\sim$ where no confusion arises. Thus $d_1 = 2.$
\[\mathcal{I}_{P_1} = (z_1) = \mathcal{I}_{P_{1,1}}.\]
Applying these variable changes to $\mathcal{I}_0$ gives $\mathcal{I}_1 = (z_1+z_3^2, z_1^2 + 2z_1z_2).$ $\max(W_1) = \frac{1}{4}$ and $\Lambda_2 = (\frac{1}{2}, \frac{1}{4} , \frac{1}{4})$. 
\[
\mathcal{I}_{P_2} =(z_1 + z_3^2),
\] 
where $\mathcal{I}_{P_{2,1}}= \mathcal{I}_{P_2}$ and $\mathcal{I}_{P_{2,2}}$ is the zero ideal. Again, consider all $\Lambda_2$-homogeneous transformations  and choose $$\tilde{z}_1 = z_1 + z_3^2 \quad \tilde{z}_j = z_j, \quad j=2,3,4,$$ which makes $\mathcal{I}_{P_2}$ to be dependent on only the variable $z_1$.  Again, we ignore the sign $\sim$. Here $d_2=2$ and 
\[
\mathcal{I}_{P_2} =(z_1),
\]
where $\mathcal{I}_{P_{2,1}}= \mathcal{I}_{P_2}$ and $\mathcal{I}_{P_{2,2}}$ is the zero ideal. We apply the new coordinates to $\mathcal{I}_1$ to get $$\mathcal{I}_2 = (z_1, z_1^2+ 2z_1z_3^2 + z_3^4 + 2z_1z_2 - 2z_2z_3^2).$$ $\max(W_2) = \frac{1}{6}$ and $\Lambda_3=(\frac{1}{2}, \frac{1}{6} , \frac{1}{6})$. 
\[\mathcal{I}_{P_3} = (z_1, -2z_2z_3^2) = (z_1, 2z_2z_3^2).\] 
Here $\mathcal{I}_{P_{3,1}}= (z_1)$ and $\mathcal{I}_{P_{3,2}} =(2z_2z_3^2)$. No $\Lambda_3$-homogeneous transformation can make $\mathcal{I}_{P_3}$ to be independent of any variables and so $d_3=0$. Thus, the multitype weight $\Lambda^{*} = \Lambda_3$ and the final leading polynomial ideal is \[\mathcal{I}_{P_3} = (z_1, z_2z_3^2).\]

\end{example}

\section{Variables Changes In the Kol\' a\v r Algorithm}
We now give an explicit construction of the polynomial transformations that are performed in the Kol\' a\v r algorithm. We begin this construction by first relating the weighted homogeneous polynomial transformations to pairs of row-column operations on the Levi matrix of a sum of squares domain. More specifically, we show that each polynomial transformation corresponds to some finite sequence of row and column operations. Lemma \ref{rcoperationi} below establishes this correspondence.

Throughout this section, we shall assume the following:

\noindent Let $M \subset \C^{n+1}$ be the boundary of a sum of squares domain defined by $\{ r <0\}$, where 
\[ r = \text{2Re}(z_{n+1}) + \sum_{j=1}^N |f_j(z_1, \dots , z_n)|^2,\]
and $f_1, \dots, f_N$ are holomorphic functions in the neighborhood of the origin. Let
\[ r_0 = \text{2Re}(z_{n+1}) + \text{P}(z,\bar{z})\]
be the defining function of the model hypersurface $M_0$ of $M,$ where $P(z,\bar z)$ is a polynomial of weighted degree 1 with respect to the multitype weight at the origin $\Lambda^{*}$ of $M.$ Let A be the $n \times n$ Levi matrix of the model $M_0 \subset \C^{n+1},$ where we ignore the contribution of the $(n+1)^{st}$ coordinate as the holomorphic polynomials in the sum of squares do not depend on it as mentioned in the introduction.

\begin{lemma} \label{rcoperationi}
Assume that the D'Angelo 1-type of the hypersurface $M$ in $\C^{n+1}$ at 0 is finite. Let $i \in \{1, \dots, n \}$ be fixed, and let $h \in \C[z]$  for $z = (z_1, \dots, z_n)$ be a nonzero monomial independent of $z_i.$ Let $h_{\ell}$ denote the derivative of $h$ with respect to the variable $z_{\ell},$ which is $\partial_{z_{\ell}} h$ with $l \in \{1, \dots, n \}\setminus{\{i \}}.$ Furthermore, let $h_{\ell}(\tau)$ denote $h_{\ell}$ with every factor of $z_l$ replaced by a factor of $\tau.$ Performing the elementary row and column operations $\text{R}_{\ell} - h_{\ell} \text{R}_i \rightarrow \text{R}_{\ell}$ and $ \text{C}_{\ell}  - \bar{h}_{\ell}\text{C}_i \rightarrow \text{C}_{\ell}$ on the Levi matrix A of $r_0$ for all variables $z_{\ell}$ in $h$ corresponds to the polynomial transformation 
\[ \tilde{z}_i = z_i + \int_0^{z_{\ell}}h_{\ell}(\tau) \ d\tau = z_i + h; \:\: \quad \tilde z_{\omega}=z_{\omega} \:\: \text{for} \:\: \omega \neq i\]
in the sense that the new matrix $\tilde{A}$ obtained after these elementary operations is Hermitian and is the Levi matrix of the new defining function of the sum of squares domain after the change of variables $z_{\omega} \to \tilde z_{\omega}$ for $\omega=1, \dots, n+1$ has taken place.
\end{lemma}

\begin{remark}
The reader should note that while only variables $z_1, \dots, z_n$ play a role in the behavior of the Levi matrix, $\C^{n+1}$ is the underlying space, so all changes of variables described in this section will take place in $\C^{n+1}$ and leave $z_{n+1}$ unchanged.
\end{remark}

\begin{proof}
Suppose that the defining function $r_0$ of the model hypersurface $M_0$ is of the form
\[  
r_0 = \text{2Re}(z_{n+1}) + \sum_{t=1}^N |g_t|^2,
\]
where $P(z,\bar z) = \sum_{t=1}^N |g_t|^2$ and $g_t = \sum_{l=1}^{b_t} m_{t,l}$ is a polynomial consisting of monomials $m_{t,l}$ in the variables $z_1, \dots, z_n$ since $P(z, \bar z)$  cannot depend on the variable $z_{n+1}.$ Let $m_{t,l} = C_{t,l} \prod_{\delta = 1}^n z_{\delta}^{{\alpha}^{t,l}_{\delta}}$ with $C_{t,l} \in \C.$ For each $t$ and for $l_1, l_2 \in \{1, \dots, b_t \},$ every monomial from the expansion of $|g_t|^2$ can be written as 
\[
m_{t,l_1}\overline{m}_{t,l_2} = C_{t,l_1} \overline{C}_{t,l_2} \prod_{\delta =1}^n z_{\delta}^{{\alpha}^{t,l_1}_{\delta}} \bar z_{\delta}^{\hat{\alpha}^{t,l_2}_{\delta}}. 
\]
By writing each term $m_{t,l_1}\overline{m}_{t,l_2}$ for all $t$ in the new coordinates, we obtain $P(z, \bar{z})$ in the new coordinates. Hence it suffices to show that applying the specified elementary row and column operations to the Levi matrix of the monomial $m_{t,l_1}\overline{m}_{t,l_2}$ corresponds to the polynomial transformation  $\tilde{z}_i = z_i + h; \:\: \tilde z_{\omega}=z_{\omega} \:\: \text{for} \:\: \omega \neq i.$

Denote by D the $(i,j,k,u)$ submatrix of the Levi matrix of the monomial $m_{t,l_1}\overline{m}_{t,l_2},$ and let $\text{D} = (d_{e \bar{\kappa}})_{e,\kappa = i,j,k,u},$ where $d_{e \bar{\kappa}}$ is given by 
\[ 
d_{e \bar{\kappa}} = C_{t,l_1} \overline{C}_{t,l_2} \alpha_{e}^{t,l_1} \hat{\alpha}_{\kappa}^{t,l_2} z_{e}^{{\alpha}^{t,l_1}_{e} -1} \bar z_{e}^{\hat{\alpha}^{t,l_2}_{e}} z_{\kappa}^{{\alpha}^{t,l_1}_{\kappa}} \bar z_{\kappa}^{\hat{\alpha}^{t,l_2}_{\kappa} - 1} \prod_{\substack{\delta = 1 \\ \delta \neq e, \kappa}}^n z_{\delta}^{{\alpha}^{t,l_1}_{\delta}} \bar z_{\delta}^{\hat{\alpha}^{t,l_2}_{\delta}}.  
 \] Let $h = C z_{a_1}^{{\beta}_{a_1}}, \dots, z_{a_s}^{{\beta}_{a_s}},$ where $C \in \C,$ $\beta = (\beta_1, \dots, \beta_n)$ is a multiindex, and $a_1,\dots, a_s \in \{1, \dots, n \}\setminus{\{i \}}.$ Now, assume that $j,k \in \{a_1, \dots, a_s\}$ with $j \neq k$ and $u \notin \{a_1, \dots, a_s\}.$ Perform the elementary operations $\text{R}_{\ell} - h_{\ell} \text{R}_i \rightarrow \text{R}_{\ell}$ and $ \text{C}_{\ell}  - \bar{h}_{\ell}\text{C}_i \rightarrow \text{C}_{\ell}$ for all variables $z_{\ell}$ in $h$ on D to get 
\[
\begin{pmatrix}
 d_{i \bar{\imath}} & d_{i \bar{\jmath}} - \bar h_j d_{i \bar{\imath}} &   d_{i \bar{k}} - \bar h_k d_{i \bar{\imath}} & d_{i \bar u} \vspace{0.5cm}\\
  d_{j \bar{\imath}} - h_j d_{i \bar{\imath}}  &d_{j \bar{\jmath}} - h_j d_{i \bar{\jmath}}- \bar h_j d_{j \bar{\imath}}+ |h_j|^2 d_{i \bar{\imath}} & d_{j \bar{k}} - h_j d_{i \bar{k}}- \bar h_k d_{j \bar{\imath}}+ h_j \bar h_k d_{i \bar{\imath}} & d_{j \bar{u}} - h_j d_{i \bar{u}}  \vspace{0.5cm} \\ 
 d_{k \bar{\imath}} - h_k d_{i \bar{\imath}} & d_{k \bar{\jmath}} - \bar{h}_jd_{k \bar{\imath}}- h_k d_{i \bar{\jmath}}+ \bar h_jh_k d_{i \bar{\imath}} & d_{k \bar{k}} - h_k d_{i \bar{k}}- \bar h_k d_{k \bar{\imath}}+ |h_k|^2 d_{i \bar{\imath}}& d_{k \bar{u}} -  h_k d_{i \bar{u}} \vspace{0.5cm}\\
 d_{u \bar{\imath}} & d_{u \bar{\jmath}} - \bar h_j d_{u \bar{\imath}} & d_{u \bar{k}} - \bar h_k d_{u \bar{\imath}} & d_{u \bar{u}}
\end{pmatrix}.
\]
The matrix above is the Levi matrix for the monomial in the new coordinates 
\[
\tilde m_{l_1} \overline{\tilde m}_{l_2} = C_{t,l_1} \overline{C}_{t,l_2} (\tilde z_i - h )^{{\alpha}^{t,l_1}_e} (\bar{\tilde z}_i - \bar h)^{\hat{\alpha}^{t,l_2}_e} \prod_{\delta \neq i} z_{\delta}^{{\alpha}^{t,l_1}_{\delta}} \bar z_{\delta}^{\hat{\alpha}^{t,l_2}_{\delta}},
\]
which is obtained after applying the polynomial transformation $\tilde{z}_i = z_i + h; \:\: \tilde z_{\omega}=z_{\omega}$ for $\omega \neq i$ to $m_{t,l_1}\overline{m}_{t,l_2}.$

\medskip We now prove that the matrix $\tilde{\text{A}}$ is Hermitian. Since the matrix A is Hermitian, we will show that applying the operation $\text{R}_{\ell} - h_{\ell} \text{R}_i \rightarrow \text{R}_{\ell}$ and $ \text{C}_{\ell}  - \bar{h}_{\ell}\text{C}_i \rightarrow \text{C}_{\ell}$ to A gives a matrix that is Hermitian as well. Let $\text{A}=(a_{k \bar{l}})_{1 \leq k,l \leq n}$ be the Levi matrix. Apart from row $\ell$ and column $\ell,$ there is no change to A, which is Hermitian. Let $a_{\ell \bar{k}}$ be an entry in row $\ell.$ Then the entry $a_{k \bar{\ell}}$ is in column $\ell$ and satisfies the property that $a_{\ell \bar{k}} = \bar{a}_{k \bar{\ell}}.$ Performing the elementary operations $\text{R}_{\ell} - h_{\ell} \text{R}_i \rightarrow \text{R}_{\ell}$ and $ \text{C}_{\ell}  - \bar{h}_{\ell}\text{C}_i \rightarrow \text{C}_{\ell}$ on A gives a new matrix $\tilde{\text{A}}$ with $a_{\ell \bar{k}} - h_{\ell} a_{i \bar{k}}$ in row $\ell$ and $a_{k \bar{\ell}}  - \bar{h}_{\ell} a_{k \bar{\imath}}$ in column $\ell$. Now 
\[ a_{\ell \bar{k}} - h_{\ell} a_{i \bar{k}} = \bar{a}_{k \bar{\ell}} - \bar{\bar{h}}_{\ell} \bar{a}_{k \bar{\imath}} = \overline{a_{k \bar{\ell}} - \bar{h}_{\ell} a_{k \bar{\imath}}}.\]
Thus, the new matrix $\tilde{\text{A}}$ is Hermitian as well.
\end{proof}

\bigskip

Proposition \ref{P1} stated in the introduction provides a general condition for the existence of an allowable polynomial transformation via the elementary row and column operations performed on the Levi matrix of a leading polynomial at some step of the Kol\' a\v r algorithm. It turns out that the restrictive definition of dependency mentioned in the introduction is a necessary and sufficient condition for the existence of allowable polynomial transformations. At this point, we give the proof of Proposition \ref{P1} and remark that the proof is constructive in the sense that we will show the allowable polynomial transformation on $P_j$ arises as a composition of polynomial transformations corresponding to elementary row and column operations on $\text{A}_{P_j}.$

\bigskip

\noindent \textbf{Proof of Proposition \ref{P1}:} Let $k$ be given, and denote by $R_k$ and $C_k$ the $k$-th row and $k$-th column of the matrix $\text{A}_{P_j}$ respectively. Suppose that  the $k$-th row of $\text{A}_{P_j}$ is dependent. This implies that $C_k$ must also be dependent since $\text{A}_{P_j}$ is Hermitian. Hence we can write both $R_k$ and $C_k$ respectively as:
\begin{equation} \label{dp2}
R_k = \sum_{\substack{l=1 \\
            l \neq k}}^n \beta^l R_l \ \text{and} \ C_k = \sum_{\substack{l=1 \\
            l \neq k}}^n \bar{\beta}^l C_l,
\end{equation}
where $\beta^l \in \C[z]$ for every $l$, $1 \leq l \leq n.$ As proven in Theorem \ref{thm1}, the leading polynomial $P_j$ is a sum of squares, and so we write $P_j = \sum_{s=1}^m |\phi^s|^2,$ where each $\phi^s$ is a nonzero polynomial with vanishing order greater than or equal to 1. Denote by $a_{k \bar t}$ the entry in the $(k, \bar t)$ position of the matrix $\text{A}_{P_j},$ for $t = 1, \dots, n.$ Therefore, for $k\neq l$ the entries in $R_k$ and $R_l$ are given by
\begin{equation} \label{dp3}
a_{k \bar t} = \sum_{s=1}^m \phi^s_k \bar{\phi}^s_t \quad \text{and} \quad a_{l \bar t} = \sum_{s=1}^m \phi^s_l \bar{\phi}^s_t,
\end{equation}
and the entries in $C_k$ and $C_l$ are given by 

\begin{equation} \label{dp4}
a_{t \bar k} = \sum_{s=1}^m \phi^s_t \bar{\phi}^s_k  \quad \text{and} \quad a_{t \bar l} = \sum_{s=1}^m \phi^s_t \bar{\phi}^s_l 
\end{equation}
respectively, where $\displaystyle \phi^s_j=\frac{\partial \phi^s}{\partial z_j}.$

We will show that there exist some elementary row and column operations that make $R_k$ to be identically zero and also make every monomial in $\text{A}_{P_j}$ to be independent of the variable $z_k.$

To that end, we see from \eqref{dp2} that every pair $\text{R}_k - \beta^l \text{R}_l; \ \text{C}_k -  \bar{\beta}^l \text{C}_l$ for which the polynomial $\beta^l$ is nonzero requires corresponding elementary row and column operations in the exact form expressed in the pair. If $\beta^l = 0$ or $R_l \equiv 0,$ then no elementary row or column operation is required. Therefore, assume that $\beta^l$ is nonzero and that $R_l$ is not identically zero. Let $\beta^l(\zeta_k)$ be $\beta^l$ with each factor of $z_k$ replaced by a factor of $\zeta_k.$ Then $\displaystyle \int_0^{z_k} \beta^l (\zeta_k) \ d \zeta_k$ must contain the variable $z_k$ together with all the other variables in $\beta^l.$ Recall that $\beta^l \in \C[z],$ so $\beta^l$ does not depend on any variable $\bar z_\nu.$ Hence we shall investigate all monomials in $\text{A}_{P_j}$ containing the variable $z_k.$  

We recall at this point that by applying the operator $\partial_{z_k}\partial_{\bar z_t}$ to $P_j,$ we obtain in row $R_k$ the derivatives of all monomials containing the variable $z_k.$ Let $a_{k \bar t}$ be an entry in row $R_k.$ Now, because $R_k$ is dependent, every monomial $u$ in $a_{k \bar t}$ arises as the product of a monomial $p$ in $\beta^l$ for some $l$ with a monomial $q$ in entry $a_{l \bar t}.$ So $u = pq,$ but since $u$ comes from differentiation by $\partial_{z_k}\partial_{\bar z_t},$ $P_j$ must contain a monomial $ m = u z_k \bar z_t.$ If $u \in \C,$ then no entries in $\text{A}_{P_j},$ except for those in $R_k,$ contain derivatives from $ m.$ If $u$ has positive vanishing order, then $u$ depends on at least one variable $z_{\nu}$   or $\bar z_\nu$ for some $\nu.$ Since $P_j$ is real-valued, it contains both $m$ and $\bar{m}.$ Therefore, without loss of generality, we can assume $u$ depends on $z_\nu;$ otherwise, we work with $\bar{u}.$ Since $u$ depends on $z_\nu,$ the entry $a_{\nu \bar t}$ in the $\nu$-th row $R_{\nu}$ contains the monomial $\partial_{z_{\nu}} \partial_{\bar z_t} m \neq 0,$ which has at least one factor of $z_k.$ We seek to eliminate all such monomials containing variable $z_k$ from the matrix $\text{A}_{P_j}.$

Set $\displaystyle \gamma^l = \int_0^{z_k} \beta^l(\zeta_k) \ d \zeta_k,$ and let $\gamma^l = \sum_{b=1}^e m^{l,b},$ where $m^{l,b}$ is a nonzero monomial containing the variable $z_k$ for all $b \geq 1.$ 
We recall from Lemma \ref{rcoperationi}  that for any nonzero monomial $m$ in the leading polynomial, if we perform the pair of elementary operations $\text{R}_{\nu} - \partial_{z_{\nu}} m \text{R}_{\ell} \rightarrow \text{R}_{\nu}$ and $ \text{C}_{\nu} -  \partial_{\bar z_{\nu}} \bar m \text{C}_{\ell} \rightarrow \text{C}_{\nu}$ for all variables $z_{\nu}$ in $m,$ then this pair corresponds to the polynomial transformation $\displaystyle \tilde z_{\ell} = z_{\ell} + \int_0^{z_\nu} \partial_{\tau} m(\tau) \ d\tau = z_{\ell} + m; \:\: \tilde z_{\omega} = z_{\omega}$ for $ \omega \neq \ell ,$ where $\partial_{\tau} m(\tau)$ is $\partial_{z_{\nu}} m$ with each factor of $z_\nu$ replaced by a factor of $\tau.$

Now, for each monomial $m^{l,b}$ in $\gamma^l, \ b= 1, \dots, e,$ we perform the elementary row and column operations $\text{R}_{\nu} - \partial_{z_{\nu}} m^{l,b} \text{R}_l \rightarrow \text{R}_{\nu}$ and $ \text{C}_{\nu} -  \partial_{\bar z_{\nu}} \bar{m}^{l,b} \text{C}_l \rightarrow \text{C}_{\nu}$ for every variable $z_{\nu}$ in $m^{l,b}.$ The composition of all of these polynomial transformations $\mathcal{S}$ is given by $\tilde{z}_l = z_l + \gamma^l$ for every $l$ such that $\beta^l \neq 0$ and $\tilde z_{\omega} = z_{\omega}$ for all $\omega \neq l,$ where $1\leq \omega\leq n+1.$ Note that \eqref{dp2} implies that $\gamma^l$ has the same weight as $z_l$ in $\Lambda_j$ because $P_j$ only contains terms of weight 1 with respect to $\Lambda_j,$ so $\mathcal{S}$ is $\Lambda_j$-homogeneous as needed.

After all the elementary row and column operations corresponding to the polynomial transformation $\mathcal{S}$ have taken place, the entries in $R_k$ are 
\begin{equation} \label{dp5}
a^{'}_{k \bar t} = \sum_{s=1}^m \phi^s_k \bar{\phi}^s_t - \sum_{l=1}^n \beta^l \Big ( \sum_{s=1}^m \phi^s_l \bar{\phi}^s_t \Big) = a_{k \bar t} - \sum_{l=1}^n \beta^l a_{l \bar t}\equiv 0
\end{equation}
as a consequence of \eqref{dp2}. A similar argument holds for the entries in $C_k,$ which we denote by $a^{'}_{t \bar k},$ namely \eqref{dp2} implies that $a^{'}_{t \bar k} \equiv 0.$ Therefore, all entries in the $k$-th row and column of the matrix $\text{A}_{P_j}$ are identically zero after the change of variables $\mathcal{S}$ has been performed. Now, assume that the leading polynomial $P_j$ still contains the variable $z_k$ after the given change of variables has been performed on it and some cancellation occurs. Since the leading polynomial $P_j$ is a sum of squares, in the expansion of $P_j$ besides the cross terms, which could possibly cancel each other, we would have at least two squares of monomials containing $z_k.$ From the above discussion, it is clear that by performing these elementary row and column operations on $\text{A}_{P_j},$ all monomials  containing the variable $z_k$ in any of its entries will have been eliminated including any contribution from those squares. Thus, $P_j$ could not possibly have contained the variable $z_k,$ so $\mathcal{S}$ is an allowable polynomial transformation with respect to the variable $z_k.$

\medskip Conversely, suppose that 
there exists an allowable polynomial transformation on $P_j$ with respect to the variable $z_k,$ and let $\mathcal{T}$ be this polynomial transformation, which we shall express as: 
\begin{equation} \label{pd1}
\tilde z_i = z_i + \gamma^i, 
\end{equation}
for $i= 1, \dots, n+1,$ where some of the $\gamma^i$ may be zero. We note here that the transformation $\mathcal{T}$ is a $\Lambda_j$-homogeneous transformation, and so $\gamma^i$ has the same weight with respect to $\Lambda_j$ as $z_i.$ Furthermore, we note that any $\Lambda_j$-homogeneous transformation can be written in this form.

We will prove that the $k$-th row $R_k$ is dependent by showing that it satisfies the condition given in \eqref{dp2}. Assume that the variable $z_k$ is contained in $\gamma^i$ for  some $i \in \{1, \dots, d\}$ with $d \leq n.$ We know that each $\tilde z_i$ corresponds to the row and column operations 
\begin{equation} \label{pd2}
\text{R}_k - \gamma^i_k \text{R}_i \rightarrow \text{R}_k \quad \text{and} \quad \text{C}_k -  \bar{\gamma}^i_k \text{C}_i \rightarrow \text{C}_k
\end{equation}
respectively, for $i= 1, \dots, d,$ where $\gamma^i_k = \partial_{z_k} \gamma^i \neq 0.$ Let $\tilde P_j$ be the leading polynomial $P_j$ after the polynomial transformation $\mathcal{T}$ is applied to it. Since $\tilde P_j$ does not contain the variable $z_k,$ the entries $\tilde h_{k \bar t}$ of $R_k$ and $\tilde h_{t \bar k}$ of $C_k$ of the matrix $\text{A}_{\tilde{P}_j}$ for all $t= 1, \dots, n$ are zero entries.

Now, by simply reversing the signs involved in the elementary operations in \eqref{pd2}, we can restore $R_k$ and $C_k$ to their previous forms before the transformation $\mathcal{T}$ was applied to $P_j.$ Hence by performing the elementary row and column operations
$$\text{R}_k + \gamma^i_k \text{R}_i \rightarrow \text{R}_k \quad \text{and} \quad  \text{C}_k + \bar{\gamma}^i_k \text{C}_i \rightarrow \text{C}_k$$ for all $i = 1, \dots, d$ on the matrix  $\text{A}_{\tilde{P}_j},$ the entries in $R_k$ and $C_k$ become 

\begin{equation} \label{pd3}
h_{k \bar t} = \sum_{i=1}^d \gamma^i_k h_{i \bar t} \quad \text{and} \quad h_{t \bar k} = \sum_{i=1}^d \bar{\gamma}^i_k h_{t \bar{\imath}},
\end{equation}
where $h_{i \bar t}$ and $h_{t \bar{\imath}}$ are the entries in the $i$-th row $R_i$ and $i$-th column $C_i$ respectively. Finally, we obtain that

\begin{equation} \label{pd4}
h_{k \bar t} = \sum_{i=1}^n \gamma^i_k h_{i \bar t} \quad \text{and} \quad h_{t \bar k} = \sum_{i=1}^n \bar{\gamma}^i_k h_{t \bar{\imath}},
\end{equation}
where $\gamma^i_k = 0$ for all $i = d+1, \dots, n.$ Thus, both $R_k$ and $C_k$ are dependent, proving the statement. 
\qed

\vspace{0.5cm}

It is important to note that for any $k,$ if the diagonal $(k, \bar k)$ entry of $A_{P_j}$ is the only nonzero entry in its $k$-th row, then the $k$-th row cannot be dependent, where $A_{P_j}$ is the Levi matrix of the leading polynomial $\text{P}_j.$ This statement holds because we are working with a sum of squares domain and $A_{P_j}$ is Hermitian.

\begin{lemma} \label{newlemma2}
Let $\Gamma$ be the set of $n \times n$ matrices with coefficients in the ring $\C[z, \bar{z}].$ Let $\text{H} \in \Gamma$ be Hermitian. For some given $i$ and $k,$ let $\text{B}$ be the matrix obtained from H after the elementary row and column operations $\text{R}_k + \alpha \text{R}_i \rightarrow \text{R}_k$ and $ \text{C}_k + \bar{\alpha} \text{C}_i \rightarrow \text{C}_k,$ for some $\alpha \in \C[z],$ are performed on it. Then $\det(\text{B}) = \det(\text{H}).$
\end{lemma}

\begin{proof}
Let E be the matrix obtained from H by the elementary row operation $\text{R}_k + \alpha \text{R}_i \rightarrow \text{R}_k.$ Then $\text{B}$ is the matrix obtained from E by the elementary column operation $ \text{C}_k + \bar{\alpha} \text{C}_i \rightarrow \text{C}_k.$ It is obvious from the properties of the determinant that $\det(\text{H}) = \det(\text{E})$ and that $\det(\text{E}) = \det(\text{B}).$
\end{proof}

\medskip

\begin{lemma} \label{newlemma3}
Assume that the D'Angelo 1-type of $M \subset \C^{n+1}$ at 0 is finite. At step $j$ of the Kol\' a\v r algorithm 
for the computation of the multitype at 0, let $\text{P}_j$ be the leading polynomial and let $\text{Q}_j$ be the leftover polynomial.

If the determinant of $\text{A}_{P_j}$ is nonzero, then $P_j$ is independent of the largest number of variables, and no polynomial transformation needs to be performed on it before the next step in the Kol\' a\v r algorithm.
\end{lemma}

\begin{proof}
We shall give a proof of the contrapositive of the statement of this lemma, which states that if there exists an allowable transformation and hence one of the rows of $\text{A}_{P_j}$ is dependent by Proposition~\ref{P1}, then the determinant of $\text{A}_{P_j}$ is zero. Suppose that $\mathscr{R} = \{R_1, \dots, R_n \},$ the set of all rows of the matrix $\text{A}_{{P}_j},$ is dependent and that none of the rows is identically equal to zero. Thus for some $k,$ we can write 
$$R_k = \sum_{\substack{l=1\\
                   l \neq k}}^n \alpha_l R_l,$$
where $\alpha_l \in \C[z] \ \text{and} \ R_l$  is the $l$-th row of $\text{A}_{P_j}.$ From the proof of Proposition \ref{P1}, we know that there must exist some elementary row and column operations that transform $\text{A}_{{P}_j}$ into the matrix $\tilde{\text{A}}_{P_j}$ whose $k$-th row and column have all zero entries. Since the matrix $\tilde{\text{A}}_{P_j}$ has at least one row with all entries equal to zero, its determinant equals zero. From Lemma \ref{newlemma2}, we know that $\det(\tilde{\text{A}}_{P_j}) = \det(\text{A}_{{P}_j}).$ 

Thus, $\det(\text{A}_{{P}_j}) = 0,$ which is the result we need.
\end{proof}

\vspace{0.5cm}
Given the way the leading polynomial $P_j$ and its Levi matrix $\text{A}_{P_j}$ are constructed, it is possible that $P_j$ could be independent of at least one of the variables. If that is the case, then the determinant of the Levi matrix corresponding to the leading polynomial $P_j$ will always be zero since it will have at least one row that is identically zero. Therefore, we need a way to determine when a subset of all nonzero rows of $\text{A}_{P_j}$ is independent. To address this situation, we shall consider the following:

\vspace{0.3cm}
Let $\text{A}_{P_j} = (a_{i \bar l})_{1 \leq i,l \leq n}$ be the Levi matrix of the leading polynomial $P_j$ at step $j$ of the Kol\' a\v r algorithm. Let $m$ be the number of nonzero rows of the matrix $\text{A}_{P_j}.$ Denote by $\text{A}_{P_j|m}$ the principal submatrix obtained from $\text{A}_{P_j}$ by removing all zero rows and columns to get precisely $m$ rows and columns, for some $m \leq  n.$ Put differently, $\text{A}_{P_j|m}$ is the submatrix consisting only of all nonzero rows and columns of $\text{A}_{P_j}.$ If $m = n,$ then none of the rows and columns are identically zero. 

Also, via the elementary row and column operations, $\text{A}_{P_j|m}$ can be transformed into a leading principal submatrix of $\text{A}_{P_j},$ where the first $m$ rows and columns are the ones that remain. In this case,  $\text{A}_{P_j|m} = (a_{i \bar l})_{1 \leq i,l \leq m}.$

Let $\mathscr{R} = \{R_1, \dots, R_n \}$ be the set of all rows of the matrix $\text{A}_{{P}_j}$ and let $\mathscr{S}= \{R_{b_1}, \dots, R_{b_m} \}$ be a subset of $\mathscr{R},$ for $b_e \in \{1, \dots, n \}, \ e = 1, \dots, m$ such that each element $R_{b_e}$ is not identically zero. Then $\mathscr{S}$ is the set of all non zero rows of the submatrix $\text{A}_{P_j|m}.$

We now restate Proposition \ref{P1} and Lemma \ref{newlemma3} as follows:

\begin{proposition} \label{newP}
Assume that the D'Angelo 1-type of $M \subset \C^{n+1}$ at 0 is finite. At step $j$ of the Kol\' a\v r algorithm 
for the computation of the multitype at 0, let $P_j$ be the leading polynomial, and let $Q_j$ be the leftover polynomial.

There exists an allowable polynomial transformation on $P_j$ with respect to the variable $z_k$ via the elementary row and column operations if and only if the $k$-th row of $\text{A}_{\text{P}_j|m}$ is dependent. 
\end{proposition}

\begin{lemma} \label{newlemma5}
Assume that the D'Angelo 1-type of $M \subset \C^{n+1}$ at 0 is finite. At step $j$ of the Kol\' a\v r algorithm 
for the computation of the multitype at 0, let $P_j$ be the leading polynomial, and let $Q_j$ be the leftover polynomial. 

If the determinant of $\text{A}_{P_j|m}$ is nonzero, then $P_j$ is independent of the largest number of variables, and no polynomial transformation needs to be performed on it before the next step in the Kol\' a\v r algorithm.
\end{lemma}

The proofs of Proposition \ref{newP} and Lemma \ref{newlemma5} are identical to the proofs given for Proposition \ref{P1} and Lemma \ref{newlemma3} respectively since the latter do not depend on rows being identically equal to zero. We also note here that the converses of 
Lemmas \ref{newlemma3} and \ref{newlemma5} do not hold. The reason is that 
the notion of dependency given in \eqref{dp2} is more restrictive than the standard notion of dependency in linear algebra, so there might not exist a row that is dependent according to our definition, but the set of rows may satisfy the standard notion of dependency, in which case the determinant of the Levi matrix would be identically equal to zero.

Now, the natural question to ask at this point is this: Given a Levi matrix of a leading polynomial with zero determinant, how can we tell whether or not it has dependent rows? Also, if there exist dependent rows, how can we identify such rows in order to determine the allowable polynomial transformations corresponding to these dependent rows? The answers to these questions lie in the formulation of an algorithm, which we will describe in the next subsection.
 
\bigskip

\subsection{Jacobian Modules and Gradient Ideals}
A study of the elementary row and column operations on the Levi matrix reveals that  a row (column) operation on the Levi matrix is performed by a multiplication on the left (right) of the Levi matrix by an elementary row (column) matrix. The Levi matrix of a sum of squares domain can always be decomposed as the product of the complex Jacobian matrix of the holomorphic functions that generate the domain and its conjugate transpose. Therefore, every row operation on the Levi matrix could be performed on the complex Jacobian matrix, while every column operation on the Levi matrix is performed on the conjugate transpose of the complex Jacobian matrix. Let A be an $n \times n$ Levi matrix of a domain given by the sum of squares of $N$ holomorphic functions. Then elementary matrices are $n \times n$ matrices, while the complex Jacobian matrix and its conjugate transpose must be $n \times N$ and $N \times n$ matrices respectively. 

In our study of elementary row and column operations performed on the Levi determinant of a sum of squares domain, one particular property of the Levi matrix of a square of one of the generators drew our attention: All entries of any given column (row) have the same anti-holomorphic (holomorphic) parts, and so a study of the relationship between these entries narrows down to a study of the relationship between their holomorphic (anti-holomorphic) parts. In other words, we expect that the study of the Levi matrix will be much easier if we transition from the sum of squares to the underlying ideal of holomorphic functions that generate the domain as we already saw was the case for the computation of the multitype.

We start with a couple of definitions that we specialize to complex polynomials since those are the objects that appear in the Kol\' a\v r algorithm when it is applied to a sum of squares domain instead of the full holomorphic generators: 

\begin{definition} \label{gradid}
Let $f \in \C[z_1, \dots, z_n]$ be polynomial in the variables $z_1, \dots, z_n$ with coefficients in $\C.$ As in \cite{NDS06}, we define the \textit{gradient ideal of f} as the ideal generated by the partial derivatives of $f:$
\begin{equation} \label{gi1}
\mathcal{I}_{grad}(f) = \left <\nabla f \right > = \left <\frac{\partial f}{\partial z_1}, \cdots, \frac{\partial f}{\partial z_n} \right >.
\end{equation}
\end{definition}

\begin{definition} \label{gradid2}
Given the ideal $\left<f\right> = \left<f_1, \dots, f_N\right> \subset \C[z_1, \dots, z_n],$ we define the \textit{Jacobian module of $f$} as 
\begin{equation} \label{gi2}
\mathfrak{J}_{\left<f\right>} = \left[\frac{\partial f}{\partial z_1}, \cdots, \frac{\partial f}{\partial z_n} \right],
\end{equation}
where each $\frac{\partial f}{\partial z_j}$ is a vector. $\mathfrak{J}_{\left<f\right>}$ is a module over the polynomial ring $\C[z_1, \dots, z_n].$
\end{definition}
\noindent To every Jacobian module $\mathfrak{J}_{\left<f\right>},$ we associate the complex Jacobian matrix $\text{J}(f)$ given by
\begin{equation} \label{ji}
\text{J}(f) = \begin{pmatrix}
\frac{\partial f_1}{\partial z_1} & \frac{\partial f_2}{\partial z_1}& \cdots & \frac{\partial f_N}{\partial z_1}\\
\vdots & \vdots & \ddots & \vdots \\
\frac{\partial f_1}{\partial z_n} & \frac{\partial f_2}{\partial z_n} & \cdots & \frac{\partial f_N}{\partial z_n}
\end{pmatrix}.
\end{equation}
Likewise, to each gradient ideal $\mathcal{I}_{grad}(f_i) = \left <\frac{\partial f_i}{\partial z_1}, \cdots, \frac{\partial f_i}{\partial z_n} \right >$ of the generator $f_i \in \C[z_1, \dots, z_n]$ of $\left<f\right>,$ we associate the $i$-th column of $\text{J}(f)$ for $1 \leq i \leq n.$ The reader should note that row operations on $\text{J}(f)$ are precisely operations on the module $\mathfrak{J}_{\left<f\right>}.$

\medskip\noindent Now, let $\left<f\right> = \left<f_1, \dots, f_N\right> \subset \C[z_1, \dots, z_n]$ be the leading polynomial ideal at some step of the Kol\' a\v r algorithm. Then we are particularly interested in simplifying the Jacobian module $\mathfrak{J}_{\left<f\right>}$ such that it is generated by the minimal number of generators. Every generator that is eliminated is a linear combination of some partial derivatives of $f$ with coefficients in $\C[z_1, \dots, z_n].$ Since every generator of the Jacobian module represents a row of the complex Jacobian matrix, every eliminated generator represents a dependent row in the complex Jacobian matrix. Owing to this connection, from every eliminated generator, we can construct a sequence of elementary row operations that corresponds to the linear combination of some partial derivatives of $f$ as described in Proposition \ref{P1}. Hence we obtain polynomial transformations corresponding to these row operations.
 
It is easy to observe this relationship if $N=1.$ Then reducing the number of generators of $\mathfrak{J}_{\left<f\right>},$ if possible, reduces the number of nonzero rows of the associated complex Jacobian matrix. Thus, the minimal number of generators required to generate the Jacobian module is precisely the number of independent rows of the complex Jacobian matrix, which is the same as the number of variables on which the corresponding leading polynomial ideal $\left<f\right>$ is dependent by Proposition \ref{P1} after the corresponding change of variables. Hence the number $d_j$ at step $j$ of the Kol\' a\v r algorithm is given by $d_j=n -\#f ,$ where $\#f$ is the minimal number of generators generating the Jacobian module $\mathfrak{J}_{\left<f\right>},$ $n$ is the number of variables in the polynomial ring $\C[z_1, \dots, z_n],$ and $d_j$ is the largest number of variables of which the leading polynomial at step $j$ is independent. This gives an algebraic characterization of the number $d_j$ in the Kol\' a\v r algorithm.

In the more general case where $N > 1,$ reducing the number of generators of the Jacobian module implies reducing the generators of all gradient ideals by the same operations. Thus, $\frac{\partial f}{\partial z_{\ell}},$ for some $\ell,$ is a generator that is eliminated in the Jacobian module $\mathfrak{J}_{\left<f\right>}$ if and only if the generator $\frac{\partial f_i}{\partial z_{\ell}}$ of the gradient ideal $\mathcal{I}_{grad}(f_i)$ for all $i = 1, \dots, N$ is eliminated, namely reduced to 0. 

Clearly, if there exists at least one gradient ideal $\mathcal{I}_{grad}(f_i)$ with minimal number of generators equal to $n,$ then the Jacobian module $\mathfrak{J}_{\left<f\right>}$ cannot have fewer than $n$ generators, i.e. no reduction via row operations is possible.

In this set-up, all elementary pairs of row-column operations on the Levi matrix reduce to just elementary row operations on the complex Jacobian matrix.

\subsection{Row Reduction Algorithm}

We shall now devise an algorithm that constructs explicitly the polynomial transformations required at each step of the Kol\' a\v r algorithm when applied to the complex Jacobian matrix of the leading polynomial ideal.

The algorithm gives the conditions for characterizing the required elementary row operations that correspond to the polynomial transformations needed in the Kol\' a\v r algorithm. The application of the algorithm to the complex Jacobian matrix corresponding to a given leading polynomial ideal will eliminate all dependent rows, if they exist, from the complex Jacobian matrix.

Now, for any $j\geq 1,$ let $\text{J}_{P_j}$ be the complex Jacobian matrix corresponding to the leading polynomial $P_j$ at step $j$ of the Kol\' a\v r algorithm.  

\medskip \noindent {\small GRADIENT IDEALS:} Let $P_j = \sum_{i=1}^N |h_i|^2.$  Then the leading polynomial ideal is given by $\mathcal{I}_{P_j} := \left <h \right > = \left <h_1, \dots, h_N \right >,$ and the gradient ideal of $h_i$ is $\mathcal{I}_{grad}(h_i) = \left <\frac{\partial h_i}{\partial z_1}, \cdots, \frac{\partial h_i}{\partial z_n} \right >, \ i = 1, \dots, N.$ Note that the complex Jacobian matrix of $P_j$ is given by $\text{J}_{P_j}=\text{J}(\mathcal{I}_{P_j})=\text{J}(h)$ and the Levi matrix $\text{A}_{P_j}$ of $P_j$ is the product of the complex Jacobian matrix of $P_j$ with its conjugate transpose $\text{J}^*(h):$ $\text{A}_{P_j} = \text{J}(h)\text{J}^*(h).$ We shall reduce, if possible, the number of generators of each gradient ideal one at a time and control the changes that occur in other gradient ideals as a result of these reduction operations. By control, we mean setting appropriate conditions on the reduction operations used. 

\medskip \noindent {\small STRUCTURE OF THE ALGORITHM:} If $\det \text{A}_{P_j}=  \det \left(\text{J}(h)\text{J}^*(h) \right)$ is nonzero, then no change of variables is required by Lemma~\ref{newlemma3}. Thus, assume that $\det \left(\text{J}(h)\text{J}^*(h)\right) = 0.$ Then:

\begin{itemize}
\item[1.] We begin the process by first considering the gradient ideal $\mathcal{I}_{grad}(h_i) = \left <\frac{\partial h_i}{\partial z_1}, \cdots, \frac{\partial h_i}{\partial z_n} \right >$ for any $i.$ Simplify the gradient ideal $\mathcal{I}_{grad}(h_i)$ such that it consists of the minimal number of generators. Assume that
\begin{equation} \label{gi4}
\frac{\partial h_i}{\partial z_k} = \sum_{u = 1}^v \gamma_{c_u} \frac{\partial h_i}{\partial z_{c_u}},
\end{equation}
for some $k, \ c_u \neq k, \ v < n,$ and $\gamma_{c_u}$ a nonzero polynomial in $\C[z_1, \dots, z_n].$ Then perform the following elementary row operations on the complex Jacobian matrix $\text{J}(h):$
\begin{equation} \label{gi5}
\text{R}_{\ell} - \frac{\partial \zeta_{c_u}}{\partial z_{\ell}} \text{R}_{c_u} \rightarrow \text{R}_{\ell},
\end{equation}
for all $u=1, \dots, v$ and for all variables $z_{\ell}$ in $\zeta_{c_u} = \int_0^{z_k} \gamma_{c_u}(t) \ dt.$ By Lemma \ref{rcoperationi}, the row operations in \eqref{gi5} correspond to the polynomial transformation given by 
\begin{equation} \label{gi6}
\tilde{z}_{c_u} = z_{c_u} + \int_0^{z_k} \gamma_{c_u}(t) \ dt; \:\: \tilde z_{\omega} = z_{\omega},
\end{equation}
for all $\omega \neq c_u$ and for all $u=1, \dots, v.$ The generator $\frac{\partial h_i}{\partial z_k}$ vanishes in $\mathcal{I}_{grad}(h_i)$ after the row operations in \eqref{gi5} are performed on $\text{J}(h).$ 

We say row $\text{R}_{c_u}$  is used as a \textit{central row} in the sequence of row operations and the generator $\frac{\partial h_i}{\partial z_{c_u}}$ is used as a \textit{central generator} in the simplification of the gradient ideal $\mathcal{I}_{grad}(h_i)$ for all $i.$ We remark here that for all subsequent row operations performed on the complex Jacobian matrix, the row $\text{R}_{c_u}$ cannot be used as a central row  and $\frac{\partial h_e}{\partial z_{c_u}}$ cannot be used as a central generator in the simplification of any other gradient ideal $\mathcal{I}_{grad}(h_e)$ for $e \neq i.$ This condition is imposed due to Proposition ~\ref{P1}, which is an equivalence. Reusing a central row or a central generator might reintroduce a variable that has been eliminated from the leading polynomial.

\item[2.] Next, consider another gradient ideal $\mathcal{I}_{grad}(h_s) = \left <\frac{\partial h_s}{\partial z_1}, \cdots, \frac{\partial h_s}{\partial z_n} \right >$ for $s \neq i.$ Clearly, the $k$-th generator of this gradient ideal is  
\begin{equation} \label{gi7}
\frac{\partial h_s}{\partial z_k} - \sum_{u = 1}^v \gamma_{c_u} \frac{\partial h_s}{\partial z_{c_u}}
\end{equation}
due to the row operations given in \eqref{gi5}. We simplify the ideal $\mathcal{I}_{grad}(h_s)$ such that it has the minimal number of generators while ensuring that the generators $\frac{\partial h_s}{\partial z_{c_u}}$ for $u = 1, \dots, v$ are not used as central generators in the simplification of $\mathcal{I}_{grad}(h_s).$ Perform the related row operations.

\item[3.] Proceed similarly by considering other gradient ideals different from the previous ones. Since there are only finitely many gradient ideals and finitely many generators that generate each of them, the process will terminate after a finite number of steps. 

\end{itemize}

\vspace{0.5cm}
We will show in the lemma that follows that the polynomial transformation in \eqref{gi6} corresponding to the row operations given in \eqref{gi5} is $\Lambda_j$-homogeneous. Thus, we state the following:

\begin{lemma} \label{htrans}
Assume that the D'Angelo 1-type of $M \subset \C^{n+1}$ at 0 is finite. At step j of the Kol\' a\v r algorithm for the computation of the multitype at $0,$ let $\Lambda_j$ be the weight, $P_j$ the leading polynomial, and $\mathcal{I}_{P_j}$ the corresponding leading polynomial ideal. Let $\mathcal{I}_{grad}(\psi)$ be the gradient ideal of some generator $\psi$ of the ideal $\mathcal{I}_{P_j}.$ Assume that 
\begin{equation} \label{ht1}
\frac{\partial \psi}{\partial z_k} = \sum_{u = 1}^v \gamma_{c_u} \frac{\partial \psi}{\partial z_{c_u}},
\end{equation}
for some $k,$ where $k \neq c_u, \ v < n,$ and $\gamma_{c_u}$ is a nonzero polynomial in $\C[z_1, \dots, z_n].$ Let $\zeta_{c_u} = \int_0^{z_k} \gamma_{c_u}(t) \ dt.$

Then the polynomial transformation given by $\tilde{z}_{c_u} = z_{c_u} + \zeta_{c_u}; \:\: \tilde z_{\omega} = z_{\omega}$ for all $\omega \neq c_u$ corresponding to the elementary row operations $\text{R}_{\ell} - \frac{\partial \zeta_{c_u}}{\partial z_{\ell}} \text{R}_{c_u} \rightarrow \text{R}_{\ell}$ for all variables $z_{\ell}$ in $\zeta_{c_u}$ and for all $u=1, \dots, v$ performed on the complex Jacobian matrix $\text{J}_{P_j}$ is $\Lambda_j$-homogeneous for all $u = 1, \dots, v.$
\end{lemma}

\begin{proof}
We start the proof by recalling from Theorem \ref{thm1} that the leading polynomial is a sum of squares at every step of the Kol\' a\v r algorithm. Hence $P_j$ is a sum of squares. Let $\Lambda_j=(\lambda_1, \dots, \lambda_n).$ Since variables are not ordered in increasing weight order, we let $\phi: \{1,\dots, n\} \to \{1,\dots, n\}$ be the bijection $\phi=(\phi_1, \dots, \phi_n)$ such that the variable $z_i, \ 1 \leq i \leq n,$ has weight $\lambda_{\phi_i}.$ 

We will show that the term $\gamma_{c_u}$ in the polynomial transformation is of weighted degree $(\lambda_{\phi_{c_u}} - \lambda_{\phi_k}).$ Let $\nu$ be the weighted degree of $\gamma_{c_u}$ with respect to the weight $\Lambda_j.$  By our hypothesis, the weighted degrees of $\frac{\partial \psi}{\partial z_{c_u}}$ and $\frac{\partial \psi}{\partial z_k} $ are $\frac{1}{2} - \lambda_{\phi_{c_u}}$ and $\frac{1}{2} - \lambda_{\phi_k}$ respectively since all generators of the leading polynomial ideal $\mathcal{I}_{P_j}$ are of weighted degree $\frac{1}{2}$ with respect to $\Lambda_j.$ The weighted degree of the right hand side of the expression given in \eqref{ht1} is $\nu + \frac{1}{2} - \lambda_{\phi_{c_u}}.$ Hence solving the equation in \eqref{ht1} for $\nu$ gives $\nu = \lambda_{\phi_{c_u}} - \lambda_{\phi_k}.$ Thus, the weighted degree  of $\zeta_{c_u} = \int_0^{z_k} \gamma_{c_u}(t) \ dt$ is $\lambda_{\phi_{c_u}}$ as required. 
\end{proof}

\begin{remark}
The polynomial $\gamma_{c_u}$ cannot depend on the variable $z_{c_u}$ because if it were to depend on $z_{c_u},$ then its weighted degree would satisfy $\nu \geq \lambda_{\phi_{c_u}},$ but $\nu = \lambda_{\phi_{c_u}} - \lambda_{\phi_k},$ which gives a contradiction because $\lambda_{\phi_k} >0.$
\end{remark}

\medskip

\begin{lemma} \label{htrans2}
Assume that the D'Angelo 1-type of $M$ in $\C^{n+1}$ at 0 is finite. At step j of the Kol\' a\v r algorithm for the computation of the multitype at $0,$ let $\text{P}_j$ be the leading polynomial, and let $\mathcal{I}_{P_j} = \left < h \right > \subset \C[z_1, \dots, z_n]$ be the corresponding leading polynomial ideal.

If the Row Reduction algorithm is applied to the complex Jacobian matrix $\text{J}(h),$ then every dependent row of $\text{J}(h)$ vanishes. In other words, the leading polynomial ideal $\mathcal{I}_{P_j}$ is independent of the largest number of variables after the Row Reduction algorithm is applied to $\text{J}(h).$
\end{lemma}

\begin{proof}
From Theorem \ref{thm1} the leading polynomial $P_j$ is a sum of squares, and so let $P_j = \sum_{i=1}^N |h_i|^2.$ Then the leading polynomial ideal $\mathcal{I}_{\text{P}_j}$ is $\left < h \right > = \left < h_1, \dots, h_N \right >,$ and let the gradient ideal of each generator $h_i$ be $\mathcal{I}_{grad}(h_i) = \left <\frac{\partial h_i}{\partial z_1}, \cdots, \frac{\partial h_i}{\partial z_n} \right >.$

Now, assume that $R_k,$ the $k$-th row of $\text{J}(h)$ is dependent. We will show that the generator $\frac{\partial h}{\partial z_k}$ of the Jacobian module given by $\mathfrak{J}_{\left<h\right>} = \left [\frac{\partial h}{\partial z_1}, \cdots, \frac{\partial h}{\partial z_n} \right ]$ vanishes after applying the Row Reduction algorithm on the complex Jacobian matrix $\text{J}(h).$ Since $R_k$ is dependent, we can write the generator $\frac{\partial h}{\partial z_k}$ as
\begin{equation} \label{ht2}
\frac{\partial h}{\partial z_k} = \sum_{u =1}^{v} \gamma_{c_u} \frac{\partial h}{\partial z_{c_u}},
\end{equation}
for some $k, \ c_u \neq k, \ v < n,$ and $\gamma_{c_u}$ a nonzero polynomial in $\C[z_1, \dots, z_n]$ for every $u.$ Hence every generator $\frac{\partial h_i}{\partial z_k}$ of the gradient ideal $\mathcal{I}_{grad}(h_i)$ can be written as
\begin{equation} \label{ht3}
\frac{\partial h_i}{\partial z_k} = \sum_{u = 1}^{v} \gamma_{c_u} \frac{\partial h_i}{\partial z_{c_u}},
\end{equation}
for $i = 1, \dots, N$ and the same polynomial coefficients $\gamma_{c_u}.$ Thus, it suffices to show that $\frac{\partial h_i}{\partial z_k}$ vanishes at the termination of the algorithm for every $i = 1, \dots, N.$ Consider the ideal $\mathcal{I}_{grad}(h_i)$ for some $i \in \{1, \dots, N \}.$ If the generator $\frac{\partial h_i}{\partial z_k}$ is zero, then there is nothing to be done, and so we move to a different gradient ideal. Hence suppose that $\frac{\partial h_i}{\partial z_k}$ is nonzero. Then at least one of the generators $\frac{\partial h_i}{\partial z_{c_u}}$ is nonzero for some $u.$ Suppose that $\frac{\partial h_i}{\partial z_{c_u}} \neq 0$ for all $u \in \{1, \dots, w\}$ for $w \leq v.$ Since it satisfies the condition in \eqref{ht3}, we perform the elementary row operations $\text{R}_{\ell} - \frac{\partial \zeta_{c_u}}{\partial z_{\ell}} \text{R}_{c_u} \rightarrow \text{R}_{\ell},$ for all variables $z_{\ell}$ in $\zeta_{c_u} = \int_0^{z_k} \gamma_{c_u}(t) \ dt$ and for all $u=1, \dots, w.$ This eliminates the term $\sum_{u = 1}^w \gamma_{c_u} \frac{\partial h_e}{\partial z_{c_u}},$ from the ideal $\mathcal{I}_{grad}(h_i).$ The generator $\frac{\partial h}{\partial z_k}$ of the Jacobian module becomes 
\begin{equation} \label{ht4}
\frac{\partial h}{\partial z_k} = \sum_{u = w+1}^v \gamma_{c_u} \frac{\partial h}{\partial z_{c_u}}
\end{equation}
after the row operations have been performed on $\text{J}(h).$ Note here that the generator $\frac{\partial h_e}{\partial z_{c_u}},$ for all $e \neq i$ and $u=1, \dots,w$ cannot be central in any simplification process in the gradient ideal $\mathcal{I}_{grad}(h_e)$ after the row operations.

Next, consider another gradient ideal $\mathcal{I}_{grad}(h_e)$ for $e \neq i.$ Then its $k$-th generator after the reduction operation is 
\begin{equation} \label{ht5}
\frac{\partial h_e}{\partial z_k} - \sum_{u = 1}^{w} \gamma_{c_u} \frac{\partial h_e}{\partial z_{c_u}} = \sum_{u = w+1}^{v} \gamma_{c_u} \frac{\partial h_e}{\partial z_{c_u}}.
\end{equation}
If the expression in \eqref{ht5} equals zero, then there is nothing left to be done. If the expression in \eqref{ht5} is nonzero, then $\frac{\partial h_e}{\partial z_{c_u}} \neq 0$ for some $u = w+1, \dots, q$ with $q \leq v.$ Perform the elementary row operations $\text{R}_{\ell} - \frac{\partial \zeta_{c_u}}{\partial z_{\ell}} \text{R}_{c_u} \rightarrow \text{R}_{\ell},$ for all variables $z_{\ell}$ in $\zeta_{c_u} = \int_0^{z_k} \gamma_{c_u}(t) \ dt$ and for all $u=w+1, \dots, q$ to eliminate the term $\sum_{u = w+1}^q \gamma_{c_u} \frac{\partial h_e}{\partial z_{c_u}},$ where $q \leq v.$ We follow this process through in each of the distinct gradient ideals until all gradient ideals have been considered. The expression in \eqref{ht4} becomes zero at some point; otherwise, we get a contradiction to $R_k$ being dependent.

\end{proof}

\begin{example} \label{newKc}
Let the hypersurface $M \subseteq \C^5 $ be given by $r=0$ with
\[r = \text{2Re}(z_{5}) + |(z_1 + z_2 z_4 )^2+ z_2^4|^2 + |(z_1 + z_2 z_3^2)^2 |^2 + |z_2^9|^2 + |z_3^{10}|^2 + |z_4^{12}|^2. \]

Let $B = 2(z_1+ z_2 z_4), \ C = 2(z_1 + z_2 z_3^2),$ and let the ideal associated to the domain $M$ be $\left <h \right >= \left < (z_1 + z_2 z_4 )^2+ z_2^2, (z_1 + z_2 z_3^2)^2, z_2^9, z_3^{10}, z_4^{12} \right >.$ The complex Jacobian matrix is given by  
\[\text{J}(h)=
\begin{pmatrix}
B & C & 0& 0 & 0 \vspace{0.5cm}\\

z_4B+4z_2^3 & z_3^2 C & 9z_2^8 & 0 & 0 \vspace{0.5cm}\\ 

0& 2z_2z_3C & 0& 10z_3^9 & 0 \vspace{0.5cm}\\

z_2B & 0 & 0& 0 &  12z_4^{11} 
\end{pmatrix}
.\]
The Bloom-Graham type is 4, $\mathcal{I}_{P_1} = \left<z_1^2, z_1^2\right>,$ and $\Lambda_1= (\frac{1}{4}, \frac{1}{4}, \frac{1}{4}, \frac{1}{4}).$ The maximum $\text{W}_1 = \frac{1}{8},$ the leading polynomial ideal $\mathcal{I}_{P_2} = \left< (z_1 + z_2 z_4 )^2+z_2^4, z_1^2 \right>,$ and $\Lambda_2= (\frac{1}{4}, \frac{1}{8}, \frac{1}{8}, \frac{1}{8}).$ The maximum $\text{W}_2 = \frac{1}{16},$ $\mathcal{I}_{P_3} = \left<(z_1 + z_2 z_4 )^2 +z_2^4, (z_1 + z_2 z_3^2 )^2\right>,$ and $\Lambda_3= (\frac{1}{4}, \frac{1}{8}, \frac{1}{8}, \frac{1}{16}).$ The complex Jacobian matrix corresponding to $\mathcal{I}_{P_3}$ is given by
\[\text{J}(h_1, h_2)=
\begin{pmatrix}
B & C  \vspace{0.5cm}\\

z_4B+4z_2^3 & z_3^2 C  \vspace{0.5cm}\\ 

0& 2z_2z_3C \vspace{0.5cm}\\

z_2B & 0  
\end{pmatrix}
.\]
Consider the gradient ideal $\mathcal{I}_{grad}(h_1) = \left < B, z_4B+4z_2^3, 0, z_2B \right >.$ Since $ \frac{\partial h_1}{\partial z_4} = z_2 \frac{\partial h_1}{\partial z_1},$ its simplification is $\mathcal{I}_{grad}(h_1) = \left < B, 4z_2^3 \right >.$ We perform the elementary row operations $\text{R}_2 - z_4 \text{R}_1 \rightarrow \text{R}_2$ and $\text{R}_4 - z_2\text{R}_1 \rightarrow \text{R}_4$ on the matrix $\text{J}(h).$ The matrices above become
\[\text{J}(h)=
\begin{pmatrix}
B & C & 0 & 0 & 0 \vspace{0.5cm}\\

4z_2^3 & (z_3^2-z_4)C & 9z_2^8 & 0 & 0 \vspace{0.5cm}\\ 

0 & 2z_2z_3C& 0 & 10z_3^9 & 0 \vspace{0.5cm}\\

0 & -z_2C& 0&  0 &  12z_4^{11} 
\end{pmatrix} \quad \text{and} \quad \text{J}(h_1, h_2)=\begin{pmatrix}
B & C  \vspace{0.5cm}\\

4z_2^3 &  (z_3^2-z_4)C  \vspace{0.5cm}\\ 

0 & 2z_2z_3C \vspace{0.5cm}\\

0 & -z_2C 
\end{pmatrix}.
\] 
These operations correspond to the polynomial transformation $\tilde z_1 = z_1 + z_2z_4: \:\: \tilde z_{\omega} = z_{\omega}$ for $\omega \neq 1,$ and now $B = 2\tilde z_1 $ and $C = 2(\tilde z_1 + \tilde z_2(\tilde z_3^2 - \tilde z_4)).$ The generator $\frac{\partial h_2}{\partial z_1}$ cannot be a central generator in the simplification of $\mathcal{I}_{grad}(h_2),$ and row 1 cannot be used as a central row in any subsequent elementary row operations. Consider the next gradient ideal $\mathcal{I}_{grad}(h_2) = \left < C, (\tilde z_3^2 - \tilde z_4)C, 2\tilde z_2 \tilde z_3C, - \tilde z_2C \right >.$ Here the generator $-\tilde z_2 C$  in the fourth component is the only central generator in $\mathcal{I}_{grad}(h_2).$ Thus, $ \frac{\partial h_2}{\partial z_3} = - 2\tilde z_3 \frac{\partial h_2}{\partial z_4},$ and its simplification is $\left < C, (\tilde z_3^2 - \tilde z_4)C, - \tilde z_2C \right >.$

We perform the elementary row operations $\text{R}_3 + 2z_3 \text{R}_4 \rightarrow \text{R}_3$ on the matrix $\text{J}(h).$ The matrices become
\[\text{J}(h)=
\begin{pmatrix}
B & C & 0 & 0 & 0 \vspace{0.5cm}\\

4z_2^3 & (z_3^2-z_4)C & 9z_2^8 & 0 & 0 \vspace{0.5cm}\\ 

0 & 0& 0 & 10z_3^9 & 24z_3z_4^{11} \vspace{0.5cm}\\

0 & -z_2C& 0&  0 &  12z_4^{11} 
\end{pmatrix} \quad \text{and} \quad \text{J}(h_1, h_2)=\begin{pmatrix}
B & C  \vspace{0.5cm}\\

4z_2^3 &  (z_3^2-z_4)C  \vspace{0.5cm}\\ 

0 & 0 \vspace{0.5cm}\\

0 & -z_2C 
\end{pmatrix}.
\] 
This operation corresponds to the polynomial transformation $\dot z_4 = \tilde z_4 - \tilde z_3^2; \:\: \dot z_{\omega} = \tilde z_{\omega}$ for $\omega \neq 4,$ and now $B = 2\dot z_1 $ and  $C = 2(\dot z_1 - \dot z_2\dot z_4).$ The leading polynomial ideal $\mathcal{I}_{P_3} = \left<\dot z_1^2 + \dot z_2^4, (\dot z_1 - \dot z_2 \dot z_4 )^2 \right>.$  The maximum $W_3 = \frac{1}{20},$ and $\Lambda_4 = \left(\frac{1}{4}, \frac{1}{8}, \frac{1}{8}, \frac{1}{20}\right).$   
\end{example}

\vspace{1cm}


\bibliography{Article}
\bibliographystyle{abbrv}
\end{document}